\documentclass[10pt]{amsart}
\usepackage{amsfonts}
\usepackage{amsmath}
\usepackage{amssymb}
\usepackage{fullpage}
\usepackage{enumitem}
\usepackage{graphicx}
\usepackage{float}

\usepackage{hyperref}
\hypersetup{
    colorlinks,
    citecolor=black,
    filecolor=black,
    linkcolor=blue,
    urlcolor=black
}

\usepackage{setspace}

\newcommand{\changeitem}{%
  \let\latexitem\item
  \renewcommand\item[1][]{\latexitem\relax{##1 } }%
}

\newcommand{\nocontentsline}[3]{}
\newcommand{\tocless}[2]{\bgroup\let\addcontentsline=\nocontentsline#1{#2}\egroup}

\newcommand{\codim}{\mbox{codim}}
\newcommand{\kernel}{\mbox{Ker}}

\newtheorem{thm}{Theorem}[section]
\newtheorem{cor}[thm]{Corollary}
\newtheorem{defn}[thm]{Definition}
\newtheorem{lemma}[thm]{Lemma}
\newtheorem{prop}[thm]{Proposition}
\theoremstyle{definition} \newtheorem{alg}[thm]{Algorithm}

\theoremstyle{remark} \newtheorem{ex}[thm]{Example}
\theoremstyle{remark} \newtheorem{remark}[thm]{Remark}
\theoremstyle{remark} \newtheorem{obs}[thm]{Observation}
\theoremstyle{remark} \newtheorem{rul}{Rule}

\theoremstyle{remark} \newtheorem{innercustomthm}{Rule}
\newenvironment{customthm}[1]
  {\renewcommand\theinnercustomthm{#1}\innercustomthm}
  {\endinnercustomthm}

\newlist{steps}{enumerate}{1}
\setlist[steps]{label=\textit{Step \arabic*:},leftmargin=*}

\begin{document}

\title{Singularities of Restriction Varieties in $OG(k, n)$}
\author{SE{\c{C}}K{\.I}N ADALI}
\date{}
\thanks{During part of the preparation of this article the author was supported by the RCN Grant 250104/F20.}

\begin{abstract}
Restriction varieties in the orthogonal Grassmannian are subvarieties of $OG(k, n)$ defined by rank conditions given by a flag that is not necessarily isotropic with respect to the relevant symmetric bilinear form. In particular,  Schubert varieties of Type B and D are examples of restriction varieties. In this paper, we introduce a resolution of singularities for restriction varieties in $OG(k, n)$, and give a description of their singular locus by studying components of the exceptional locus of the resolution.
\end{abstract}

\maketitle
\tableofcontents

\section{Introduction}

In this paper, we present a resolution of singularities $\pi$ for restriction varieties in $OG(k, n)$. We also give a method for the description of the singular locus, and show that it is equal to the image of the exceptional locus of $\pi$ in most cases.
\medskip

Let $Q$ be a non-degenerate symmetric bilinear form on a vector space $W$ of dimension $n$ over the complex numbers. Let  $k_1<\cdots<k_h$ be positive integers such that $2k_h \leq n$. Let $F(k_1, \ldots, k_h; n)$ be the ordinary flag variety, and let $OF(k_1, \ldots, k_h; n)$ be the orthogonal partial flag variety parameterizing  subspaces
\[ W_1 \subseteq \cdots \subseteq W_h \]
of $W$ isotropic with respect to $Q$, where $W_i$ has dimension $k_i$. A restriction variety is the intersection of $OF(k_1, \ldots, k_h; n)$ with a Schubert variety in $F(k_1, \ldots, k_h; n)$ defined by a flag satisfying certain tangency conditions with respect to $Q$. Orthogonal Schubert varieties are examples of restriction varieties when the flag is isotropic.
\medskip

Restriction varieties have found applications in the restriction problem in cohomology. The inclusion $i : OF(k_1, \ldots, k_h; n) \hookrightarrow F(k_1, \ldots, k_h; n)$ induces $i^\ast : H^\ast(F(k_1, \ldots, k_h; n)) \to H^\ast(OF(k_1, \ldots, k_h; n))$, and given a Schubert class $\tau$ in $H^\ast(F(k_1, \ldots, k_h; n))$, we would like to express $i^\ast(\tau)$ as a non-negative linear combination of the Schubert classes in $H^\ast(OF(k_1, \ldots, k_h; n))$. The rule used to compute the cohomology class of a restriction variety solves this problem \cite{coskun1}. Similarly, \emph{symplectic restriction varieties} can be used to solve the same problem for the inclusion $i : SF(k_1, \ldots, k_h; n) \hookrightarrow F(k_1, \ldots, k_h; n)$, see \cite{coskun2, coskun3}. There are also applications to the rigidity problem. Restriction varieties give explicit deformations of Schubert varieties in certain instances, and hence show that the corresponding classes are not rigid. This paper studies the singularities of restriction varieties in $OG(k, n)$. We introduce a resolution of singularities, and study its exceptional locus. This resolution is inspired by the Bott-Samelson/Zelevinsky resolution for Schubert varieties, but has a more intricate construction reflecting the richer geometry of restriction varieties. We also describe the singular locus explicitly, and give a criterion for when it is equal to the image of the exceptional locus of $\pi$.
\medskip

Let $F_Q$ be the quadratic polynomial associated to $Q$. A $k$-plane $\Lambda$ is isotropic with respect to $Q$ if and only if its projectivization is contained in the quadric hypersurface defined by $F_Q$. The \emph{orthogonal Grassmannian} $OG(k, n)$ parameterizes $k$-dimensional subspaces of $W$ that are isotropic with respect to $Q$. Equivalently, this is the Fano variety of $(k-1)$-planes contained in a quadric hypersurface in $\mathbb{P}W$.
\medskip

Restriction varieties in the orthogonal Grassmannian are subvarieties of $OG(k, n)$ that parameterize isotropic subspaces of $W$ with respect to a flag that is not necessarily isotropic. Let $Q^{r}_{d}$ be a quadratic form of corank $r$ obtained by restricting $Q$ to a vector space of dimension $d$. Let $L_e$ denote an $e$-dimensional subspace that is isotropic with respect to $Q$. A restriction variety $V$ in $OG(k, n)$ is defined in terms of a sequence
\[ L_{n_1} \subseteq \ldots \subseteq L_{n_s} \subseteq Q^{r_{k-s}}_{d_{k-s}} \subseteq \ldots \subseteq Q^{r_1}_{d_1} \; . \]

$V$ parameterizes $k$-dimensional isotropic linear spaces that intersect $L_{n_j}$ in a subspace of dimension $j$ for all $1 \leq j \leq s$ and $Q^{r_i}_{d_i}$ in a subspace of dimension $k-i+1$ for all $1 \leq i \leq k-s$. There are two important conditions we impose on these sequences: The first is that we want the isotropic linear spaces $L_{n_j}$ and the singular loci of sub-quadrics $Q^{r_i, sing}_{d_i}$ to be in the most special position. This is ensured by the conditions
\begin{itemize}
\item $Q^{r_{i-1}, sing}_{d_{i-1}} \subseteq Q^{r_i, sing}_{d_i}$ and
\item $ \dim\left(L_{n_j} \cap Q^{r_i, sing}_{d_i}\right)= \min(n_j, r_i)$ for every $1 \leq j \leq s$ and $1 \leq i \leq k-s$.
\end{itemize}
\smallskip

In accordance with this positioning we require the $k$-planes parameterized by $V$ to intersect $Q^{r_i}_{d_i}$ in a certain way. Let $x_i$ be the number of isotropic linear spaces $L_{n_j}$ of the sequence contained in $Q^{r_i, sing}_{d_i}$. We require the $(k-i+1)$-dimensional subspace contained in $Q^{r_i}_{d_i}$ to intersect $Q^{r_i, sing}_{d_i}$ in a subspace of dimension $x_i$.

Secondly, we require the sub-quadrics to be irreducible. This is reflected in the condition
\begin{itemize}
\item $r_{k-s} \leq d_{k-s} -3$.
\end{itemize}
\smallskip

Informally we can think of restriction varieties as subvarieties of $OG(k, n)$ that interpolate between Schubert varieties in $OG(k, n)$, which is associated to maximal tangency conditions for the linear sections of the quadric hypersurface $Q$, and restrictions of general Schubert varieties in $G(k, n)$ to $OG(k, n)$ which is associated to minimal tangency conditions.
\medskip

The main results of this paper are the following:
\smallskip

\theoremstyle{plain} \newtheorem{res1}{$\quad$}
\begin{res1}
Theorem \ref{resthm} gives a resolution of singularities $\pi$ for restriction varieties.
\end{res1}

The resolution of singularities we introduce is inspired by the Bott-Samelson/Zelevinsky resolution for Schubert varieties. In order to resolve singularities, we construct a resolution that makes use of maximal dimensional isotropic linear subspaces at each step of the sequence. The resolution is constructed using a tower of Grassmannian and orthogonal Grassmannian bundles. We show that images of the components of the exceptional  locus of $\pi$ which have codimenson larger than 1 are in the singular locus. We study the tangent space of a restriction variety at a point for the images of the remaining components, and get a complete description of the singular locus.

\newtheorem{res2}[res1]{$\quad$}
\begin{res2}
Corollary \ref{thecor} describes the singular locus of a restriction variety in $OG(k, n)$.
\end{res2}

We give a method for finding the singular locus of a restriction variety in $OG(k, n)$ that is based on our study of the exceptional locus of $\pi$. In particular, this method presents an alternative to the method of describing the singular locus of a Schubert variety of Type B or D by checking for smoothness at each orbit.

\newtheorem{res3}[res1]{$\quad$}
\begin{res3}
Given a restriction variety $V$ in $OG(k, n)$, Algorithm \ref{alg} gives the singular locus of $V$.
\end{res3}

The organization of this paper is as follows: In Section 2, we review the well-known results on the singularities of Schubert varieties in $G(k, n)$. In Section 3, we review the necessary background and the definition of restriction varieties. In Section 4, we define the resolution of singularities, and study its exceptional locus. In Section 5, we present the algorithm for the singular locus and conclude with some examples.
\bigskip

\section{Singularities of Schubert Varieties in $G(k, n)$}

In this section, we introduce a language for Schubert varieties in the Grassmannian that will generalize to restriction varieties in a straight-forward way. This section not only serves as a reminder of some classical results on Schubert varieties, but also underlines some ideas used in the following sections. We refer the reader to \cite{billeylakshmibai} for an extensive exposition on the singularities of Schubert varieties.
\medskip

In our definition, we use sequences whose steps correspond to rank conditions giving the Schubert variety. Let $W$ be an $n$-dimensional vector space over the complex numbers, and consider $G(k, W)=G(k, n)$, the Grassmannian of $k$-planes on $W$. We define a Schubert variety $\Sigma$ in $G(k, n)$ in terms of a fixed complete flag, that is, a nested sequence of subspaces
\[ 0 \subseteq W_1 \subseteq \cdots \subseteq W_{n-1} \subseteq W_n = W \]
with $\dim W_i = i$. Consider a subsequence $W_\bullet$ of length $k$:
\[ W_{n_1} \subseteq \cdots \subseteq W_{n_k} \; .\]
The Schubert variety $\Sigma$ associated to $W_\bullet$ is defined as the closure of the locus
\[ \Sigma(W_\bullet)^0 = \left\{  \Lambda \in G(k, n) \; \middle| \; \dim \left( \Lambda \cap W_{n_i} \right) = i \; \mbox{for all} \; 1 \leq i \leq k \right\} . \]
If there are steps in $W_\bullet$ with consecutively increasing dimensions, certain conditions are implied by the others, and the number of rank conditions necessary to define the Schubert variety is less than the number of steps in the sequence. In order to define Schubert varieties in a concise way, we introduce \textit{partitions}.

\begin{defn}
Given a sequence of increasing positive integers $n_1, \ldots, n_k$, let $n_{a_1}, \ldots, n_{a_t}$ be the subsequence such that $n_{a_g}+1 \neq n_{a_g+1}$, and let
\[ \alpha_g= \Big| \left\{ \mbox{the indices $n_i$ that occur in} \; W_\bullet \; \big| \; n_i \leq n_{a_g}, \; a_g-i=n_{a_g}-n_i \right\} \Big| \; \; \mbox{for all} \; \; 1 \leq g \leq t \;. \]
Then the data $( n_{a_1}^{\alpha_1}, \ldots, n_{a_t}^{\alpha_t} )$ is defined to be the partition associated to the sequence $n_1, \ldots, n_k$.
\end{defn}

In other words, $n_{a_g}$ is the largest dimensional integer in each group of consecutive integers, and $\alpha_g$ counts the integers in that group. Note that we have $a_g=\sum_{i=1}^g\alpha_i$ and $a_t=k$. The Schubert variety $\Sigma$ in $G(k, n)$ associated to the partition $( n_{a_1}^{\alpha_1}, \ldots, n_{a_t}^{\alpha_t} )$ is defined as the closure of the locus
\[ \Sigma^0 = \left\{ \Lambda \in G(k, n) \; \middle| \; \dim\left( \Lambda \cap W_{n_{a_l}} \right)=a_l \;\; \mbox{for all} \;\; 1 \leq l \leq t \right\} \; . \]
Being homogeneous under the action of $GL(n)$, the open cell $\Sigma^0$ is smooth.

\begin{ex}
Let $\Sigma$ be the Schubert variety in $G(7, 17)$ given by the sequence
\[ W_2 \subseteq W_6 \subseteq W_7 \subseteq W_{11} \subseteq W_{12} \subseteq W_{13} \subseteq W_{15} \; . \]
This variety is defined as the closure of the locus
\begin{align*}
\Sigma^0 = \{ \Lambda \in G(7, 17) \; | \; & \dim ( \Lambda \cap W_2 ) = 1, \;\; \dim ( \Lambda \cap W_7 ) = 3, \\
& \dim ( \Lambda \cap W_{13} ) = 6, \;\; \dim ( \Lambda \cap W_{15} ) =7 \} \; .
\end{align*}
The partition associated to this Schubert variety is $( 2^1, 7^2, 13^3, 15^1 )$.
\end{ex}
\smallskip

The following proposition recalls the dimension of a Schubert variety using the sequence and the partition notations.
\begin{prop} \label{Sdim}
The dimension of a Schubert variety $\Sigma$ in $G(k, n)$ associated to the sequence $W_\bullet : \; W_{n_1} \subseteq \cdots \subseteq W_{n_k}$ or the partition $( n_{a_1}^{\alpha_1}, \ldots, n_{a_t}^{\alpha_t} )$ is given by
\[ \dim \Sigma = \sum_{i=1}^k (n_i-i) = \sum_{l=1}^t \alpha_l(n_{a_l}-a_l) \; .\]
\end{prop}

\bigskip


Schubert varieties in the Grassmannian admit a natural resolution $\pi : \widetilde{\Sigma} \to \Sigma$ such that the image of the exceptional locus of $\pi$ is equal to the singular locus of $\Sigma$. Let $\Sigma$ be given by the partition $( n_{a_1}^{\alpha_1}, \ldots, n_{a_t}^{\alpha_t} )$ and let $\widetilde{\Sigma}$ be the Schubert variety in the flag variety $F(a_1, \ldots, a_t; n)$ defined by
\[ \widetilde{\Sigma} = \left\{ (T^1, \ldots, T^t) \in F(a_1, \ldots, a_t; n) \; \middle| \; T^l \subseteq W_{n_{a_l}} \;\; \mbox{for all} \;\; 1 \leq l \leq t \right\} \;. \]
Since $\widetilde{\Sigma}$ is an iterated tower of Grassmannians, it is smooth and irreducible. The natural projection ${\pi : F(a_1, \ldots, a_t; n) \to G(k, n)}$ given by $(T^1, \ldots, T^t) \mapsto T^t$ maps $\widetilde{\Sigma}$ onto $\Sigma$ and the map is injective over the smooth open cell $\Sigma^0$. The inverse image $\pi^{-1}(\Lambda)$ of a general point $\Lambda \in \Sigma^0$ is determined uniquely as
\[ T^l = \Lambda \cap W_{n_{a_l}}\; , \;\; 1 \leq l \leq t \; . \]
By Zariski's Main Theorem, $\pi$ is an isomorphism over $\Sigma^0$ and hence a resolution of singularities of $\Sigma$.
\smallskip

The map has positive dimensional fibers over the locus of $k$-planes $\Lambda$ with the property that $\dim ( \Lambda \cap W_{n_{a_l}} ) > a_l$ for some $1 \leq l \leq t-1$. Let $\Sigma_{s_l}$ be the closure of the locus
\[ \Sigma_{s_l}^0= \left\{ \Lambda \in G(k, n) \; \middle| \; \dim \left( \Lambda \cap W_{n_{a_l}} \right) = a_l +1, \; \dim \left( \Lambda \cap W_{n_{a_g}} \right) = a_g \; \mbox{for all} \; 1 \leq g \leq t, \; i \neq l \right\} \; \]
for some $1 \leq l \leq t-1$.
\smallskip

The exceptional locus of $\pi$ consists of the union of the inverse images of $\Sigma_{s_l}$ for all $1 \leq l \leq t-1$. Let us study the codimension of the components of the exceptional locus of $\pi$. Over each $\Sigma_{s_l}$, the inverse image $\Sigma_{s_l}$ is irreducible of codimension
\[ \codim \left( \pi^{-1}(\Sigma_{s_l}) \right) = \codim \left( \Sigma_{s_l} \right) - \dim \left( \pi^{-1}(\Lambda) \right) \]
for a general $\Lambda \in \Sigma_{s_l}$. By Proposition \ref{Sdim} we have
\begin{align*}
\codim \left( \Sigma_{s_l} \right) & =  \alpha_l (n_{a_l}-a_l) + \alpha_{l+1}(n_{a_{l+1}}-a_{l+1}) \\
& \quad\quad - (\alpha_l +1)(n_{a_l}-a_l-1) - (\alpha_{l+1} -1)(n_{a_{l+1}}-a_{l+1}) \\
& = n_{a_{l+1}} - n_{a_l} - (a_{l+1} - a_l) + \alpha_l +1 \; .
\end{align*}
On the other hand, for a general $\Lambda \in \Sigma_{s_l}$ we have
\begin{align*}
\pi^{-1}(\Lambda) = \{ (T^1, \ldots, T^t) \; | \; &  T^g=\Lambda \cap W_{n_{a_g}} \;\; \mbox{for all} \;\; 1 \leq g \leq t, \;\; g \neq l  \\
& \mbox{and} \;\; T^{l-1} \subseteq T^l \subseteq \Lambda \cap W_{n_{a_l}} \}
\end{align*}
So, for an element of $\pi^{-1}(\Lambda)$, the coordinate $T^l$ is the only one that is not determined uniquely and it can be parameterized by $G(a_l-a_{l-1}, a_l+1-a_{l-1})$. This Grassmannian has dimension $a_l-a_{l-1}=\alpha_l$. Therefore we have
\[ \codim \left( \pi^{-1}(\Sigma_{s_l}) \right) = n_{a_{l+1}} - n_{a_l} - (a_{l+1} - a_l) +1 \geq 2  \]
since $n_{a_{l+1}} - n_{a_l} \geq a_{l+1} - a_l +1$.
\smallskip

This shows that a component of the exceptional locus of $\pi$ has codimension larger than 1. This observation with the following lemma determines the singular locus of a Schubert variety.

\begin{lemma} \label{coslem}
(\cite{coskun4}, Lemma 2.3)
Let $f : X \to Y$ be a birational morphism from a smooth, projective variety $X$ onto a normal projective variety $Y$ . Assume that $f$ is an isomorphism in codimension one. Then $p \in Y$ is a singular point if and only if $f^{-1}(p)$ is positive dimensional.
\end{lemma}

\begin{cor}
The image of the exceptional locus of the resolution of singularities $\pi : \widetilde{\Sigma} \to \Sigma$ is equal to the singular locus of $\Sigma$.
\end{cor}
\smallskip

\begin{ex}
For the Schubert variety given by the partition $( 2^1, 7^2, 13^3, 15^1 )$, the variety $\widetilde{\Sigma}$ is defined as
\begin{align*}
\widetilde{\Sigma} = \{ (T^1, T^2, T^3, T^4) \in F(1, 3, 6, 7; 17) \;\; | \;\; & T^1 \subseteq W_2, \;\; T^2 \subseteq W_7, \\
& T^3 \subseteq W_{13}, \;\; T^4 \subseteq W_{15} \} \; .
\end{align*}
The projection $\pi : (T^1, T^2, T^3, T^4) \mapsto T^4$ maps $\widetilde{\Sigma}$ onto $\Sigma$. The exceptional locus consists of the union of the inverse images of the closures of the following loci:
\begin{align*}
\Sigma^0_{s_1} = \{ \Lambda \in G(7, 17) \; | \; & \dim ( \Lambda \cap W_2 ) = 2, \;\; \dim ( \Lambda \cap W_7 ) = 3, \\
& \dim ( \Lambda \cap W_{13} ) = 6, \;\; \dim ( \Lambda \cap W_{15} ) =7 \} \; .
\end{align*}
\begin{align*}
\Sigma^0_{s_2} = \{ \Lambda \in G(7, 17) \; | \; & \dim ( \Lambda \cap W_2 ) = 1, \;\; \dim ( \Lambda \cap W_7 ) = 4, \\
& \dim ( \Lambda \cap W_{13} ) = 6, \;\; \dim ( \Lambda \cap W_{15} ) =7 \} \; .
\end{align*}
\[ \Sigma^0_{s_3} = \{ \Lambda \in G(7, 17) \; | \; \dim ( \Lambda \cap W_2 ) = 1, \;\; \dim ( \Lambda \cap W_7 ) = 3, \;\; \dim ( \Lambda \cap W_{13} ) = 7 \} \; . \]
Consequently the singular locus of the Schubert variety $\Sigma$ is given by
\[ \Sigma^{sing} = \Sigma_{s_1} \cup \Sigma_{s_2} \cup \Sigma_{s_3} \; . \]
\end{ex}
\smallskip

\begin{remark}
The subvarieties $\Sigma_{s_l}$ of the Schubert variety $\Sigma$  correspond to the hooks in the Young diagram of $\Sigma$.
\end{remark}

\bigskip

\section{Preliminaries on Restriction Varieties}

In this section, we review the definition of restriction varieties and their basic properties. Restriction varieties in $OG(k, n)$ parameterize isotropic $k$-planes that intersect elements of a fixed flag in specified dimensions. The flag does not need to be isotropic but there are some conditions imposed by basic facts about quadrics. Additionally, there are some conditions we impose to ensure that the isotropic linear spaces and the singular loci of quadrics are in the most special position. We review these conditions, and refer the reader to \cite{coskun1} for a detailed discussion.
\smallskip

Let $W$ be an $n$-dimensional vector space and let $Q$ a non-degenerate symmetric bilinear form on $W$. We recall the following basic facts about quadrics:

\begin{itemize} \label{qproperties}
\item \textbf{The corank bound.} Let $Q^{r_2}_{d_2} \subseteq Q^{r_1}_{d_1}$ be two linear sections of $Q$ such that the singular locus of $Q^{r_1}_{d_1}$ is contained in the singular locus of $Q^{r_2}_{d_2}$. Then $d_2+r_2 \leq d_1 +r_1$.
\item \textbf{The linear space bound.} The largest dimensional isotropic linear space with respect to a quadratic form $Q^r_d$ has dimension $\left\lfloor  \frac{d+r}{2}  \right\rfloor$. A linear space of dimension $j$ intersects the singular locus of $Q^r_d$ in a subspace of dimension at least $\max \left( 0, j- \left\lfloor  \frac{d-r}{2}  \right\rfloor \right)$.
\item \textbf{Irreducibility.} A sub-quadric $Q^{d-2}_d$ of $Q$ is reducible and equal to the union of two linear spaces of (vector space) dimension $d-1$ meeting along a linear space of dimension $d-2$. If $n=2k$, then the linear spaces constituting $Q^{k+1}_{k-1}$ belong to two distinct connected componens.
\item \textbf{The variation of tangent spaces.} Let a quadric $Q^r_d$ be singular along a codimension $j$ linear subspace $M$ of a linear space $L$. Then the image of the Gauss map of $Q^r_d$ restricted to the smooth points of $L$ has dimension at most $j-1$. In other words, the tangent spaces to $Q^r_d$ along the smooth points of $L$ vary at most in a $(j-1)$-dimensional family.
\end{itemize}

Let $F_Q$ denote the quadratic polynomial associated to $Q$. Let $L_{n_j}$ be an isotropic linear space of vector space dimension $n_j$. If $2n_j=n$, we denote isotropic linear spaces in different connected components as $L_{n_j}$ and $L_{n_j}^\prime$. Let $Q^{r_i}_{d_i}$ denote a sub-quadric of corank $r_i$ cut out by a $d_i$-dimensional linear section of $Q$ and denote this linear space by $\overline{Q^{r_i}_{d_i}}$. Let $F_{Q^r_d}$ denote the restriction of $F$ to $\overline{Q^{r_i}_{d_i}}$ so that $Q^r_d$ is given by the zero locus of $F_{Q^r_d}$. We denote the singular locus of $Q^{r_i}_{d_i}$ by $Q^{r_i, sing}_{d_i}$. We use the same notation for projectivizations contained in $\mathbb{P}W$. For convenience let $r_0=0$ and $d_0=n$. 
\smallskip

We use sequences of the form
\[ L_{n_1} \subseteq \ldots \subseteq L_{n_s} \subseteq Q^{r_{k-s}}_{d_{k-s}} \subseteq \ldots \subseteq Q^{r_1}_{d_1} \]
consisting of isotropic linear spaces $L_{n_j}$ and sub-quadrics $Q^{r_i}_{d_i}$ of $Q$ to define restriction varieties. The restriction variety $V$ defined via this sequence parameterizes $k$-dimensional isotropic linear spaces that intersect $L_{n_j}$ in a subspace of dimension $j$ for all $1 \leq j \leq s$ and $Q^{r_i}_{d_i}$ in a subspace of dimension $k-i+1$ for all $1 \leq i \leq k-s$. Let $x_i$ be the number of isotropic linear spaces $L_{n_j}$ of the sequence contained in $Q^{r_i, sing}_{d_i}$. We require the $(k-i+1)$-dimensional subspace $\Lambda \cap Q^{r_i}_{d_i}$ of a $k$-plane $\Lambda$ contained in $Q^{r_i}_{d_i}$ to intersect $Q^{r_i, sing}_{d_i}$ in a subspace of dimension $x_i$.

\begin{defn} \label{admissible}
A sequence of linear spaces and quadrics $(L_\bullet, Q_\bullet)$ associated to $OG(k, n)$ is an \textit{admissible sequence} if the following conditions are satisfied.
\begin{enumerate}
\item $2n_s \leq d_{k-s}+r_{k-s}$.
\item $2(k-i+1) \leq r_i +d_i$ for every $1 \leq i \leq k-s$.
\item $r_{i+1} + d_{i+1} \leq r_i + d_i \leq n$ for every $1 \leq i \leq k-s$.
\item $Q^{r_{i-1}, sing}_{d_{i-1}} \subseteq Q^{r_i, sing}_{d_i}$ for every $1 \leq i \leq k-s$.
\item $ \dim\left(L_{n_j} \cap Q^{r_i, sing}_{d_i}\right)= \min(n_j, r_i)$ for every $1 \leq j \leq s$ and $1 \leq i \leq k-s$.
\item For every $1 \leq i \leq k-s$ either $r_i=r_1=x_1$ or $r_l -r_i \geq l-i-1$ for every $l>i$. Furthermore, if $r_l=r_{l-1}>x_1$ for some $l$, then $d_i-d_{i+1}=r_{i+1}-r_i$ for all $i \geq l$ and $d_{l-1}-d_l=1$.
\item $r_{k-s} \leq d_{k-s} -3$.
\item For every $1 \leq i \leq k-s$,
\[ x_i \geq k-i+1 - \frac{d_i - r_i}{2} . \]
\item For any $1 \leq j \leq s$, there does not exist $1 \leq i \leq k-s$ such that $n_j-r_i=1$.
\end{enumerate}
\end{defn}

\smallskip

\begin{remark}
Conditions (1), (2) and (3) follow from the corank bound. In conditions (4) and (5), we require the isotropic linear spaces and the singular loci of sub-quadrics to be in the most special position. This gives a motivation for counting the sub-quadrics $Q^{r_i}_{d_i}$ from the right; the singular loci form a nested sequence of subspaces $Q^{r_1, sing}_{d_1} \subseteq \ldots \subseteq Q^{r_{k-s}, sing}_{d_{k-s}}$. Condition (6) is a technical condition that puts a restriction on the singular loci of the sub-quadrics in the sequence; it disallows a sudden gap between $Q^{r_i, sing}_{d_i}$. Condition (7) reflects the irreducibility property of quadrics. Condition (8) is a result of the linear space bound and the special positioning of the isotropic linear spaces and the singular loci of sub-quadrics. Finally, condition (9) follows from the variation of tangent spaces property of quadrics.
\end{remark}

\smallskip

\begin{defn} \label{restdef}
Let $(L_\bullet, Q_\bullet)$ be an admissible sequence for $OG(k, n)$. A restriction variety $V(L_\bullet, Q_\bullet)$ is the subvariety of $OG(k, n)$ defined as the closure of
\begin{align*}
V^0 \left( L_\bullet, Q_\bullet \right) = \big\{ \; \Lambda \in OG(k, n) \; \; \big| \; \; & \dim\left(\Lambda \cap L_{n_j}\right) =j, \; \; 1 \leq j \leq s, \\
& \dim\left(\Lambda \cap Q^{r_i}_{d_i}\right)=k-i+1, \; \; \dim\left(\Lambda \cap Q^{r_i, sing}_{d_i}\right)=x_i, \; \; 1 \leq i \leq k-s \; \big\} .
\end{align*}
\end{defn}

\smallskip

\begin{ex}
Schubert varieties in $OG(k, n)$ are restriction varieties defined via a sequence satisfying \linebreak $d_i+r_i=n$ for all $1 \leq i \leq k-s$, that is, when the quadrics in the sequence are as singular as possible. The restriction of a general Schubert variety in $G(k, n)$ to $OG(k, n)$ is also a restriction variety associated to a sequence with $s=0$ and $r_i=0$ for all $1 \leq i \leq k-s$. Hence, restriction varieties interpolate between the restrictions of Schubert varieties in $G(k, n)$ to $OG(k, n)$ and Schubert varieties in $OG(k, n)$.
\end{ex}

\smallskip

When the inequality $x_i \geq k-i+1 - \frac{d_i - r_i}{2}$ is an equality for an index $i$, then the $\frac{d_i+r_i}{2}$-dimensional linear spaces in $Q^{r_i}_{d_i}$ form two irreducible components.
\smallskip
\begin{ex} $V$ defined by
\[ Q^0_3 \subseteq Q^0_4 \]
in $OG(2, 5)$ parameterizes lines on a smooth quadric surface $Q^0_4$ in $\mathbb{P}^3$ and consists of two irreducible components. \end{ex}
\smallskip
The $(k-i+1)$-dimensional subspaces contained in $Q^{r_i}_{d_i}$ may be distinguished by their parity of the dimension of their intersection with linear spaces in each of these components.
\smallskip

\begin{defn} \label{def3}
Let $(L_\bullet, Q_\bullet)$ be an admissible sequence. An index $1 \leq i \leq k-s$ such that
\[ x_i = k-i+1 - \frac{d_i - r_i}{2} \]
is called a special index. For each special index, a marking $m_\bullet$ of $(L_\bullet, Q_\bullet)$ designates one of the irreducible components of $\frac{d_i + r_i}{2}$-dimensional linear spaces of $Q^{r_i}_{d_i}$ as even and the other one as odd, such that
\begin{itemize}
\item If $d_{i_1} + r_{i_1} = d_{i_2}+r_{i_2}$ for two special indices $i_1 < i_2$ and the component containing a linear space $\Gamma$ is designated even for $i_2$, then the component containing $\Gamma$ is designated even for $i_1$ as well; and
\item If $2n_s = d_i+r_i$ for a special index $i$, then the component to which $L_{n_s}$ belongs is assigned the parity of $s$; and
\item If $n=2k$, $m_\bullet$ assigns the component containing $L_k$ the parity that characterizes the component $OG(k, 2k)$.
A marked restriction variety $V(L_\bullet, Q_\bullet, m_\bullet)$ is the Zariski closure of the subvariety of $V^0(L_\bullet, Q_\bullet)$ parameterizing $k$-dimensional isotropic subspaces $W$, where, for each special index $i$, $W$ intersects subspaces of dimension $\frac{d_i+r_i}{2}$ of $Q^{r_i}_{d_i}$ designated even (respectively, odd) by  $m_\bullet$ in a subspace of even (respectively, odd) dimension.
\end{itemize}
\end{defn}

\bigskip


Partitions can be used to define restriction varieties using only the conditions that are not automatically satisfied as a result of others. 

\begin{defn}
Given a restriction variety $V$ in $OG(k, n)$ defined by the admissible sequence
\[ L_{n_1} \subseteq \ldots \subseteq L_{n_s} \subseteq Q^{r_{k-s}}_{d_{k-s}} \subseteq \ldots \subseteq Q^{r_1}_{d_1} \; , \]
let $(n_{a_1}, \ldots, n_{a_t})$ be the partition for $n_1, \ldots, n_s$, and let $(d_{b_1}, \ldots, d_{b_u})$ be the partition for $d_{k-s}, \ldots, d_1$. Then the data 
\[  (n_{a_1}^{\alpha_1}, \ldots, n_{a_t}^{\alpha_t}), (d_{b_1}^{\beta_1}, \ldots, d_{b_u}^{\beta_u}), (r_1, \ldots, r_{k-s})  \]
is defined to be the partitions associated to $V$.
\end{defn}

\begin{remark} For a group of sub-quadrics whose dimensions are consecutive integers, the coranks are not necessarily consecutive. In other words, a partition for $r_1, \ldots, r_{k-s}$ may contain more than $u$ entries. However, by condition (6) in the Definition \ref{admissible}, the value of $x_i$ is fixed for each group of sub-quadrics with consecutive dimensions. \end{remark}

\begin{remark} We have $a_g=\sum_{l=1}^g\alpha_l$ and $k-b_h+1=s+\sum_{l=1}^h\beta_l$ for every $1\leq g \leq t$ and $1\leq h \leq u$. The restriction variety defined by the partitions $(n_{a_1}^{\alpha_1}, \ldots, n_{a_t}^{\alpha_t}), (d_{b_1}^{\beta_1}, \ldots, d_{b_u}^{\beta_u}), (r_1, \ldots, r_{k-s})$ parameterizes $k$-dimensional isotropic linear spaces that satisfy
\[ \dim\left(\Lambda \cap L_{n_{a_g}}\right)=a_g, \;\; \dim\left(\Lambda \cap Q^{r_{b_h}}_{d_{b_h}}\right)=k-b_h+1, \;\; \mbox{and} \;\; \dim\left(\Lambda \cap Q^{r_{b_h}, sing}_{d_{b_h}}\right)=x_{b_h} \; , \]
where $1 \leq g \leq t$ and $1 \leq h \leq u$. \end{remark}

\begin{ex} To the sequence $ L_2 \subseteq L_3 \subseteq L_6 \subseteq L_7 \subseteq L_8 \subseteq Q^{18}_{11} \subseteq Q^{17}_{12} \subseteq Q^{13}_{18} $ we associate the partitions $  (3^2, 8^3), (12^2, 18), (13, 17, 18). $ \end{ex}
\smallskip

We recall the definition of a restriction variety in the next proposition using both sequence and partition notations.

\begin{prop} \label{dim}
\emph{(\cite{coskun1}, Prop 4.16)} The marked restriction variety $V(L_\bullet, Q_\bullet, m_\bullet)$ associated to a marked admissible sequence is an irreducible variety of dimension
\begin{align*}
\dim \left( V(L_\bullet, Q_\bullet, m_\bullet) \right) 
& = \sum_{j=1}^s \left( n_j-j \right) + \sum_{i=1}^{k-s} \left( d_i +x_i -2(k-i+1) \right) \\
&= \sum_{g=1}^t \alpha_g \left( n_{a_g} - a_g \right) + \sum_{h=1}^u \beta_h\left(d_{b_h}+x_{b_h}-2(k-b_h+1) + \frac{\beta_h-1}{2}\right)
\end{align*}
\end{prop}
\smallskip

Note that this expression does not depend on the marking $m_\bullet$. The restriction variety $V(L_\bullet, Q_\bullet)$ has an irreducible component for every marking $m_\bullet$ and every irreducible component of $V(L_\bullet, Q_\bullet)$ has this dimension.

\begin{ex}
The restriction variety $\Big[ L_6 \subseteq L_7 \subseteq L_8 \Big]$ is the Grassmannian $G(3, 8)$ which parameterizes planes contained in a projective space of dimension 7. It is given by $(8^3), (), ()$ in terms of partitions, and has dimension $\alpha_1(n_{a_1}-a_1)=3(8-3)=15$.
\end{ex}
\smallskip

\begin{ex}
The restriction variety $\Big[ Q^4_{11} \subseteq Q^3_{12} \subseteq Q^2_{13} \Big]$ is the Fano variety of planes contained in a quadric 11-fold in $\mathbb{P}^{12}$ singular along a line. In terms of partitions this is given by $(), (13^3), (2, 3, 4)$ and has dimension $\beta_1(d_{b_1}+x_{b_1}-2(3)+\frac{\beta_1-1}{2})=3(13+0-6+1)=24$.
\end{ex}
\smallskip

\begin{ex}
The restriction variety $\Big[ L_2 \subseteq L_3 \subseteq Q^7_{17} \subseteq Q^6_{18} \Big]$ parameterizes 3-dimensional projective linear spaces that are contained in a quadric hypersurface in $\mathbb{P}^{17}$ of corank 6 and that intersect a plane contained in the singular locus of the quadric along a line. In terms of partitions this variety is given by $(3^2), (18^2), (6, 7)$ and has dimension $\alpha_1(n_{a_1}-a_1)+\beta_1(d_{b_1}+x_{b_1}-2(5)+\frac{\beta_1-1}{2})=2(3-2)+2(18+2-8+\frac{1}{2})=27$.
\end{ex}

\bigskip


\section{The Resolution of Singularities}

In this section, we present a resolution of singularities for restriction varieties in $OG(k, n)$. We first illustrate the resolution on a few examples and then introduce the general definition.

\begin{ex} Let $V$ be the restriction variety in $OG(1, n)$ defined by the one-step sequence $Q^4_{11}$, where $n \geq 15$. This variety is a singular quadric contained in a 10-projective dimensional linear space whose singular locus is isomorphic to $\mathbb{P}^3$. Consider the flag variety $\widetilde{V}$ defined by
\[ \widetilde{V} = \big\{ (T, Z) \in OF(1, 5; n) \; | \; Q^{4, sing}_{11} \subseteq Z \subseteq Q^4_{11} \big\} \subseteq OG(1, n)\times OG(5, n).\]
The second projection map $\pi_2: (T, Z) \mapsto Z$ maps $\widetilde{V}$ onto $\big\{ Z \in OG(5, n) \; | \; Q^{4, sing}_{11} \subseteq Z \subseteq Q^4_{11} \big\}$ which is isomorphic to $OG(1, 7)$. Over such $Z$, the map has fibers $G(1, 5)$ of dimension 4 so $\widetilde{V}$ is irreducible of dimension 9. The first projection map $\pi_1: (T, Z) \mapsto T$ maps $\widetilde{V}$ onto $V$ where the inverse image is determined uniquely over the smooth locus of $V$. By Zariski's theorem, $\pi_1: \widetilde{V} \to V$ is a resolution of singularities for $V$ where the image of the exceptional locus gives the singularities of $V$. Note that, in this case $\pi_1$ is the blowup of the quadric $V=Q^4_{11}$ along its singular locus.
\end{ex}
\smallskip

\begin{ex} Let $V= \Big[ L_7 \subseteq Q^4_{11} \Big]$ contained in $OG(2, n)$, where $n \geq 15$. The variety $V$ parameterizes the lines in a singular quadric intersecting a fixed linear space that contains the singular locus of the quadric. Consider the variety defined by
\[ \widetilde{V} = \big\{ (T^1, T^2, O, Z) \; | \; T^1 \subseteq T^2, \;\; Q^{4, sing}_{11} \subseteq O \subseteq Z, \;\; T^1 \subseteq O \subseteq L_7 \;\; \mbox{and} \;\; T^2 \subseteq Z \subseteq Q^4_{11} \big\}  \]
where $\dim T^j=j$, $\dim O=5$ and $\dim Z =6$. The properties defining the variety $\widetilde{V}$ can be visualized by the diagram in Figure \ref{figure1}:
\begin{figure} \caption{$\widetilde{V}$ for $\Big[ L_7 \subseteq Q^4_{11} \Big]$ } \label{figure1}
\begin{equation*}
\input{figure1}
\end{equation*}
\end{figure}
Consider the following forgetful maps:
\[ (T^1, T^2, O, Z) \mapsto (T^1, O, Z) \mapsto (T^1, O) \mapsto (O) . \]
We show $\widetilde{V}$ is an iterated tower of $G(l, n)$ and $OG(l, n)$ bundles via these maps. The linear space $O$ satisfies $Q^{4, sing}_{11} \subseteq O \subseteq L_7$ and hence can be parameterized by $G(5-4, 7-4)=G(1, 3)$. For fixed $O$, the linear space $T^1$ satisfies $T^1 \subseteq O$ and hence can be parameterized by $G(1, 5)$. On the other hand, $Z$ satisfies $O \subseteq Z \subseteq Q^4_{11}$. Since $Z$ has to lie in the orthogonal complement of $O$, $Z$ is contained in a quadric of projective dimension 8 with a singular locus of projective dimension 4. Thus $Z$ can be parameterized by $OG(1, 5)$. Finally, the linear space $T^2$ satisfies $T^1 \subseteq T^2 \subseteq Z$ and hence can be parameterized by $G(1, 5)$. Thus $\widetilde{V}$ is a tower of the discussed $G(1, 3)$, $G(1, 5)$, $OG(1, 5)$ and $G(1, 5)$ bundles. This also shows that $\widetilde{V}$ is irreducible of dimension 13. The second projection map
\[ \pi: (T^1, T^2, O, Z) \mapsto T^2 \]
maps $\widetilde{V}$ onto $V$ with fibers determined uniquely over a general point $\Lambda$, that is, over $V^0$. The map $\pi: \widetilde{V} \to V$ is a resolution of singularities by Zariski's theorem.
\end{ex}
\smallskip

\begin{ex} Let $V = \Big[ L_5 \subseteq Q^7_{10} \subseteq Q^2_{20} \Big]$ contained in $OG(3, n)$, where $n \geq 22$. For this restriction variety we consider $\widetilde{V}$ defined by
\begin{align*}
\widetilde{V}=\big\{ (T^1, T^2, T^3, O^1, O^2, Z^1, Z^2) \; | \; & Q^{2, sing}_{20} \subseteq O^1 \subseteq O^2 \subseteq Z^2, \;\; Q^{7, sing}_{10} \subseteq Z^1, \\
& T^1 \subseteq O^1 \subseteq L_5, \;\; T^2 \subseteq O^2 \subseteq Q^7_{10}  \;\; \mbox{and} \;\;   T^3 \subseteq Z^2 \subseteq Q^2_{20} \big\}
\end{align*}
where $\dim T^j=j$, $\dim O^1=3$, $\dim O^2=4$, $\dim Z^1=8$ and $\dim Z^2=5$. The corresponding diagram is given in Figure \ref{figure2}.
\begin{figure} \caption{$\widetilde{V}$ for $\Big[ L_5 \subseteq Q^7_{10} \subseteq Q^2_{20} \Big]$ } \label{figure2}
\begin{equation*}
\input{figure2}
\end{equation*}
\end{figure}
We consider the following forgetful maps:
\begin{align*}
(T^1, T^2, T^3, O^1, O^2, Z^1, Z^2) & \mapsto (T^1, T^2, O^1, O^2, Z^1, Z^2) \mapsto (T^1, T^2, O^1, O^2, Z^1) \\
& \mapsto (T^1, O^1, O^2, Z^1) \mapsto (T^1, O^1, Z^1) \mapsto (T^1, O^1) \mapsto (O^1) .
\end{align*}
The linear space $O^1$ is parameterized by $G(1, 3)$ and for fixed $O$, $T^1$ is parameterized by $G(1, 3)$. The linear space $Z^1$ is parameterized by $OG(1, 3)$. For fixed $Z^1$, $O^2$ satisfies $O^1 \subseteq O^2 \subseteq Z^1$ and hence can be parameterized by $G(1, 5)$. Then $T^2$ is parameterized by $G(1, 3)$. In the last row, as $O^2 \subseteq Z^2 \subseteq Q^2_{20}$, $Z^2$ is parameterized by $OG(1, 14)$. Then $T^3$ is parameterized by $G(1, 3)$. Thus $\widetilde{V}$ is a tower of the discussed $G(1, 3)$, $G(1, 3)$, $OG(1, 3)$, $G(1, 5)$, $G(1, 3)$, $OG(1, 14)$ and $G(1, 3)$ bundles. Thus $\widetilde{V}$ is an irreducible smooth variety of dimension 25. The third projection map
\[ \pi: (T^1, T^2, T^3, O^1, O^2, Z^1, Z^2) \mapsto T^3 \]
gives the resolution of singularities in this example.
\end{ex}
\smallskip

\begin{ex} Let us consider the restriction variety in $OG(10, 70)$ given by the sequence
\[ L_2 \subseteq L_6 \subseteq L_{13} \subseteq L_{14} \subseteq L_{19} \subseteq Q^{17}_{30} \subseteq Q^{11}_{40} \subseteq Q^8_{45} \subseteq Q^7_{46} \subseteq Q^3_{50}. \]
In this case $\widetilde{V}$ satisfies the diagram in Figure \ref{figure3}. The dimensions of the $T$, $Z$ and $O$'s are noted as subscripts.
\begin{figure} \caption{$\widetilde{V}$ for $\Big[ L_2 \subseteq L_6 \subseteq L_{13} \subseteq L_{14} \subseteq L_{19} \subseteq Q^{17}_{30} \subseteq Q^{11}_{40} \subseteq Q^8_{45} \subseteq Q^7_{46} \subseteq Q^3_{50} \Big]$ } \label{figure3}
\begin{equation*}
\input{figure3}
\end{equation*}
\end{figure}
The variety $\widetilde{V}$ is a tower of $G(k, n)$ and $OG(k, n)$ bundles via 25 successive forgetful maps in this case. Starting with an element of $\widetilde{V}$, the forgetful maps trail each row from left to right going from the bottom row to the top row.
\end{ex}

\begin{ex} As a final example, let us illustrate $\widetilde{V}$ for the restriction variety $V = \Big[ L_7 \subseteq Q^5_9 \subseteq Q^4_{10} \Big]$ contained in $OG(3, n)$, where $n \geq 14$, with a given marking $m_\bullet$ for the special index 1. The variety $\widetilde{V}$ satisfies the diagram in Figure \ref{figure4}.
\begin{figure} \caption{$\widetilde{V}$ for $\Big[ L_7 \subseteq Q^5_9 \subseteq Q^4_{10} \Big]$ } \label{figure4}
\begin{equation*}
\input{figure4}
\end{equation*}
\end{figure}
By considering the forgetful maps
\[ (T^1, T^2, O, Z) \mapsto (T^1, O, Z) \mapsto (T^1, O) \mapsto (O) , \]
we obtain that $\widetilde{V}$ is a tower of $G(1, 3), G(1, 5), OG(2, 4)$ and $G(1, 6)$. Here $Z_7$, which satisfies $O_5 \subseteq Z_7 \subseteq Q^4_{10}$, is parameterized by $OG(2, 4)$, and the component that contains $Z$ is the component determined by the marking $m_\bullet$ of $V$.
\end{ex}
\medskip

Let us fix terminology before giving the definition. In the following we say a sequence $A=\Big[A_1 \subseteq \ldots \subseteq A_k\Big]$ is contained in a sequence $B=\Big[B_1 \subseteq \ldots \subseteq B_k\Big]$ if $A_i \subseteq B_i$ for all $1 \leq i \leq k$. We will denote by $A$ both the sequence $\Big[A_1 \subseteq \ldots \subseteq A_k\Big]$ and the ordered set $(A_1, \ldots, A_k)$.
\smallskip

Let $V(L_\bullet, Q_\bullet)$ be a restriction variety defined by the sequence
\[ L_{n_1} \subseteq \ldots \subseteq L_{n_s} \subseteq Q^{r_{k-s}}_{d_{k-s}} \subseteq \ldots \subseteq Q^{r_1}_{d_1} , \]
or equivalently, by the partitions $ (n_{a_1}^{\alpha_1}, \ldots, n_{a_t}^{\alpha_t}), (d_{b_1}^{\beta_1}, \ldots, d_{b_u}^{\beta_u}), (r_1, \ldots, r_{k-s}) $. For each $Q^{r_{b_h}}_{d_{b_h}}$, let $V(Q^{r_{b_h}}_{d_{b_h}} )$ be the subsequence consisting of isotropic linear subspaces $L_{n_{a_\theta}}$ and sub-quadrics $Q^{r_{b_\theta}}_{d_{b_\theta}}$ that strictly contain $Q^{r_{b_h}, sing}_{d_{b_h}}$ and are strictly contained in $Q^{r_{b_h}}_{d_{b_h}}$. We introduce a subsequence $O(Q^{r_{b_h}}_{d_{b_h}} )$ of the same length contained in $V(Q^{r_{b_h}}_{d_{b_h}} )$ that consists of isotropic linear subspaces $O$.
\begin{equation*}
\begin{array}{cccccccccc}
V(Q^{r_{b_h}}_{d_{b_h}}): \quad \quad & \cdots & \subseteq & L_{n_{a_\theta}} & \subseteq & \cdots & \subseteq & Q^{r_{b_\theta}}_{d_{b_\theta}} & \subseteq & \cdots \\
\rotatebox[origin=c]{90}{$\subseteq$} \quad \quad \quad & & & \rotatebox[origin=c]{90}{$\subseteq$} & & & & \rotatebox[origin=c]{90}{$\subseteq$} & & \\
O(Q^{r_{b_h}}_{d_{b_h}}): \quad \quad & \cdots & \subseteq & O^{h, n_{a_\theta}} & \subseteq & \cdots & \subseteq & O^{h, d_{b_\theta}} & \subseteq & \cdots
\end{array}
\end{equation*}

Also, the subsequence $\Big[ L_1 \subseteq \ldots \subseteq Q^{r_{b_{u-1}}}_{d_{b_{u-1}}} \Big]$ obtained by omitting the last $\beta_u$ sub-quadrics from the defining sequence will have a crucial role in the following definition.
\smallskip

Define:
\begin{align*}
\widetilde{V}(L_\bullet, Q_\bullet) := \Big \lbrace \; \big( & T^1, \ldots, T^{t+u}, \; Z^1, \ldots, Z^u, \; O( Q^{r_{b_1}}_{d_{b_1}} ), \ldots, O( Q^{r_{b_u}}_{d_{b_u}}) \big) \; \Big| \\
& \quad Q^{r_{b_h}, sing}_{d_{b_h}} \subseteq O( Q^{r_{b_1}}_{d_{b_1}} ) \subseteq Z^h \subseteq Q^{r_{b_h}}_{d_{b_h}},\\
& \quad O^{h, n_{a_\theta}} \subseteq L_{n_{a_\theta}} \;  \mbox{for all} \; L_{n_{a_\theta}} \; \mbox{in} \; V( Q^{r_{b_h}}_{d_{b_h}}),\\
& \quad O^{h, n_{a_\theta}} \subseteq O^{h+1, n_{a_\theta}} \;  \mbox{for all} \; L_{n_{a_\theta}} \; \mbox{that lies in both} \; V( Q^{r_{b_h}}_{d_{b_h}}) \; \mbox{and} \; V( Q^{r_{b_{h+1}}}_{d_{b_{h+1}}}) ,\\
& \quad O^{h, r_{b_\theta}} \subseteq Q^{r_{b_\theta}}_{d_{b_\theta}} \;  \mbox{for all} \; Q^{r_{b_\theta}}_{d_{b_\theta}} \; \mbox{in} \; V( Q^{r_{b_h}}_{d_{b_h}}),\\
& \quad O^{h, r_{b_\theta}} \subseteq O^{h+1, r_{b_\theta}} \;  \mbox{for all} \; Q^{r_{b_\theta}}_{d_{b_\theta}} \; \mbox{that lies in both} \; V( Q^{r_{b_h}}_{d_{b_h}}) \; \mbox{and} \; V( Q^{r_{b_{h+1}}}_{d_{b_{h+1}}}) ,\\
& \quad T^1 \subseteq \ldots \subseteq T^{t+u} \quad \mbox{for all} \; 1 \leq g \leq t \; \mbox{and} \; 1 \leq h \leq u \\
& \quad ( \; T^1, \ldots, T^{t+u-1} \; ) \subseteq \Big[ L_1 \subseteq \ldots \subseteq Q^{r_{b_{u-1}}}_{d_{b_{u-1}}} \Big] \; \mbox{and} \; T^{t+u} \subseteq Z^u \; \Big\rbrace
\end{align*}
where $\dim T^g=a_g$, $\dim T^{t+h}=k-b_h+1$, $\dim Z^h=r_{b_h}+(k-b_h+1)-x_{b_h}$, $\dim O^{h, n_{a_\theta}} = r_{b_h}+a_\theta-x_{b_h}$ and $\dim O^{h, r_{b_\theta}} = r_{b_h}+(k-b_\theta+1)-x_{b_h}$ for all $1 \leq g \leq t$ and $1 \leq h \leq u$.

\medskip


Drawing a diagram, as in the examples above, provides a tidier framework and gives the intuition behind this construction. Let $L_{n_{a_1}} \subseteq \ldots \subseteq L_{n_{a_\omega}}$ be the isotropic linear subspaces in the defining sequence contained in $Q^{r_{b_u}, sing}_{d_{b_u}}$, thus contained in all other $Q^{r_{b_h}, sing}_{d_{b_h}}$, $1 \leq h \leq u$. The defining properties of $\widetilde{V}$ are visualized in the diagram in Figure \ref{figure5}. Here, the linear spaces $O^{h, \bullet}$ that lie in the column of $Q^{r_{b_h}}_{d_{b_h}}$ form the sequence $O(Q^{r_{b_h}}_{d_{b_h}})$ in the definition of $\widetilde{V}$ above.

\begin{figure} \caption{Definition of $\widetilde{V}$} \label{figure5}
\begin{equation*}
\input{figure5}
\end{equation*}
\end{figure}

\bigskip

Let $\widetilde{V}(L_\bullet, Q_\bullet, m_\bullet)$ be the variety obtained by considering $\widetilde{V}(L_\bullet, Q_\bullet)$ with the marking for each special index inherited from $V(L_\bullet, Q_\bullet, m_\bullet)$. There is a natural projection from $\widetilde{V}(L_\bullet, Q_\bullet, m_\bullet)$ to $V(L_\bullet, Q_\bullet, m_\bullet)$ given by
\[ \pi \; : \; \big(  T^1, \ldots, T^{t+u}, \; Z^1, \ldots, Z^u, \; O( Q^{r_{b_1}}_{d_{b_1}} ), \ldots, O( Q^{r_{b_u}}_{d_{b_u}}) \big) \; \mapsto \; T^{t+u}. \]


\begin{prop} \label{prop2}
Let $V(L_\bullet, Q_\bullet, m_\bullet)$ be a marked restriction variety. The variety $\widetilde{V}(L_\bullet, Q_\bullet, m_\bullet)$ associated to $V(L_\bullet, Q_\bullet, m_\bullet)$ is a smooth irreducible variety of the same dimension as $V(L_\bullet, Q_\bullet, m_\bullet)$.
\end{prop}

\begin{proof}
Consider the successive forgetful maps omitting one coordinate of $\widetilde{V}$ at a time, going from left to right in each row, starting at the bottom row and going up. The proof of this proposition is based on constructing a tower of $G(l, n)$ and $OG(l, n)$ bundles via these forgetful maps. In the following, we study the following possible types of rows in a diagram:

\begin{enumerate}
\setlength\itemsep{1.5em}
\vspace{0.5em}

\item
For $L_{n_{a_g}} \subsetneq Q^{r_{b_u}}_{d_{b_u}}$, we have $T^{g-1} \subseteq T^g \subseteq L_{n_{a_g}}$. Hence $T^g$ is parameterized by $G(a_g-a_{g-1}, n_{a_g}-a_{g-1})$ which has dimension $(a_g-a_{g-1})(n_{a_g}-a_g)=\alpha_g(n_{a_g}-a_g)$.

\item
Suppose for $L_{n_{a_g}}$, the sub-quadrics whose singular loci lie between $L_{n_{a_g}}$ and $L_{n_{a_{g-1}}}$ are $Q^{r_{b_{\eta + c}}}_{d_{b_{\eta + c}}}, \ldots, Q^{r_{b_\eta}}_{d_{b_\eta}}$ for some number $c$, that is,
\[ Q^{r_{b_u}, sing}_{d_{b_u}} \subseteq \ldots \subseteq Q^{r_{b_{\eta +c+1}}, sing}_{d_{b_{\eta +c+1}}} \subsetneq L_{n_{a_{g-1}}} \subseteq Q^{r_{b_{\eta +c}}, sing}_{d_{b_{\eta +c}}} \subseteq \ldots \subseteq Q^{r_{b_\eta}, sing}_{d_{b_\eta}} \subsetneq L_{n_{a_g}} . \]
Note that $x_{b_\eta}= \ldots = x_{b_{\eta+c}} = a_{g-1}$ in this setting. The row consisting of $T^g$, $O^{\bullet, n_{a_g}}$, $L_{n_{a_g}}$ satisfies the diagram in Figure \ref{figure6}
\begin{figure} \caption{Row of Type 2 in $\widetilde{V}$} \label{figure6}
\begin{equation*}
\input{figure6}
\end{equation*}
\end{figure}
\smallskip

We start by choosing $O^{\eta, n_{a_g}}$. The linear space $O^{\eta, n_{a_g}}$ satisfying $Q^{r_{b_\eta}, sing}_{d_{b_\eta}} \subseteq O^{\eta, n_{a_g}} \subseteq L_{n_{a_g}}$ is parameterized by the Grassmannian $G((r_{b_\eta}+a_g-x_{b_\eta})-r_{b_\eta}, n_{a_g}-r_{b_\eta})$. In a similar fashion, the parameterization of $T^g, O^{u, n_{a_g}}, \ldots, O^{\eta, n_{a_g}}$ are given by Grassmannians whose dimensions add up to $\alpha_g(n_{a_g}-a_g)$ as in Table \ref{table1}.
\begin{table} \caption{Dimensions for a row of Type 2 in $\widetilde{V}$} \label{table1}
\begin{equation*}
\begin{array}{lcl}
\mbox{\emph{Coordinates of $\widetilde{V}$ in the $g$-th row:}} &\quad& \mbox{\emph{Dimensions of the corresponding Grassmannian:}}\\
Q^{r_{b_\eta}, sing}_{d_{b_\eta}} \subseteq O^{\eta, n_{a_g}} \subseteq L_{n_{a_g}} && (a_g-x_{b_\eta})\big(n_{a_g}-a_g-(r_{b_\eta}-x_{b_\eta})\big)\\
Q^{r_{b_{\eta+1}}, sing}_{d_{b_{\eta+1}}} \subseteq O^{\eta+1, n_{a_g}} \subseteq O^{\eta, n_{a_g}} &&  (a_g-x_{b_{\eta+1}})\big((r_{b_\eta}-x_{b_\eta})-(r_{b_{\eta+1}}-x_{b_{\eta+1}})\big)\\
\rotatebox[origin=c]{-90}{$\cdots$} && \rotatebox[origin=c]{-90}{$\cdots$}\\
Q^{r_{b_{\eta+c}}, sing}_{d_{b_{\eta+c}}} \subseteq O^{\eta+c, n_{a_g}} \subseteq O^{\eta+c-1, n_{a_g}} &&  (a_g-x_{b_{\eta+c}})\big((r_{b_{\eta+c-1}}-x_{b_{\eta+c-1}})-((r_{b_{\eta+c}}-x_{b_{\eta+c}})\big)\\
O^{\eta+c+1, n_{a_{g-1}}} \subseteq O^{\eta+c+1, n_{a_g}} \subseteq O^{\eta+c, n_{a_g}} &&  (a_g-a_{g-1})\big((r_{b_{\eta+c}}-x_{b_{\eta+c}})-((r_{b_{\eta+c+1}}-x_{b_{\eta+c+1}})\big)\\
\rotatebox[origin=c]{-90}{$\cdots$} && \rotatebox[origin=c]{-90}{$\cdots$}\\
O^{u, n_{a_{g-1}}} \subseteq O^{u, n_{a_g}} \subseteq O^{u-1, n_{a_g}} &&  (a_g-a_{g-1})\big((r_{b_{u-1}}-x_{b_{u-1}})-((r_{b_{u}}-x_{b_{u}})\big)\\
T^{g-1} \subseteq T^g \subseteq O^{u, n_{a_g}} && (a_g-a_{g-1})\big( r_{b_u} -x_{b_u} \big)
\end{array}
\end{equation*}
\end{table}

\item
Consider the row that corresponds to $Q^{r_{b_1}}_{d_{b_1}}$. Depending on $r_{b_1}$, there are two possibilities for the diagram. If $r_{b_1} \geq n_{a_t}$ then $Z^1$ is determined by $Q^{r_{b_1}, sing}_{d_{b_1}} \subseteq Z^1 \subseteq Q^{r_{b_1}}_{d_{b_1}}$. Explicitly, suppose $L_{n_{a_t}}$ is positioned as $Q^{r_{b_{c+1}}, sing}_{d_{b_{c+1}}} \subsetneq L_{n_{a_t}} \subseteq Q^{r_{b_c}, sing}_{d_{b_c}} \subseteq \ldots \subseteq Q^{r_{b_1}, sing}_{d_{b_1}}$ for some number $c$. Note that $x_{b_1}=\ldots=x_{b_c}=t$ in this setting. The diagram is as in Figure \ref{figure7}.
\begin{figure} \caption{Row of Type 3 in $\widetilde{V}$} \label{figure7}
\begin{equation*}
\input{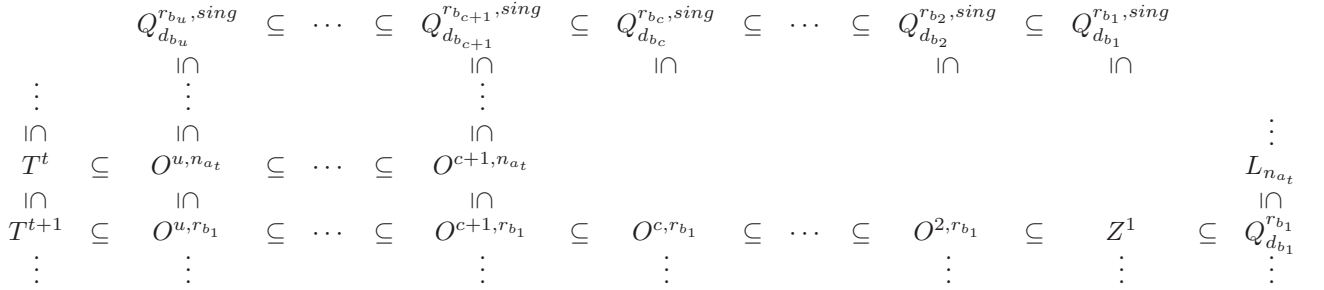}
\end{equation*}
\end{figure}
We start by choosing $Z^1$. The linear space $Z^1$ satisfies $Q^{r_{b_1}, sing}_{d_{b_1}} \subseteq Z^1 \subseteq Q^{r_{b_1}}_{d_{b_1}}$ and $\dim Z^1 = r_{b_1}+(k-b_1+1)-x_{b_1}=r_{b_1}+\beta_1$. Hence $Z^1$ can be parameterized by $OG(\beta_1, d_{b_1}-r_{b_1})$. Note that $OG(\beta_1, d_{b_1}-r_{b_1})$ is irreducible as $d_{b_1}-r_{b_1}-2\beta_1 \geq 3$. The linear spaces $T^{t+1}, O^{u, r_{b_1}}, \ldots, O^{2, r_{b_1}}$ can be parameterized by Grassmannians whose dimensions add up to $\beta_1(d_{b_1}+x_{b_1}-2(k-b_1+1)-\frac{\beta_1-1}{2})$, see Table \ref{table2}. Note that $\dim OG(k, n)=k(n-2k+\frac{k-1}{2})$ (see \cite{coskun1} for a proof).
\begin{table}  \caption{Dimensions for a row of Type 3 in $\widetilde{V}$} \label{table2}
\begin{equation*}
\begin{array}{lll}
\mbox{\emph{Coordinates of $\widetilde{V}$ in the $(t+1)$-st row:}} &\quad& \mbox{\emph{Dimensions of the corresponding Grassmannian:}}\\
Q^{r_{b_1}, sing}_{d_{b_1}} \subseteq Z^1 \subseteq Q^{r_{b_1}}_{d_{b_1}} && \beta_1(d_{b_1}+x_{b_1}-2(k-b_1+1) -(r_{b_1}-x_{b_1})+\frac{\beta_1-1}{2})\\
Q^{r_{b_2}, sing}_{d_{b_2}} \subseteq O^{2, r_{b_1}} \subseteq Z^1 && (k-b_1+1-x_2)((r_{b_1}-x_{b_1})-(r_{b_2}-x_{b_2}))\\
Q^{r_{b_3}, sing}_{d_{b_3}} \subseteq O^{3, r_{b_1}} \subseteq O^{2, r_{b_1}} && (k-b_1+1-x_3)((r_{b_2}-x_{b_2})-(r_{b_3}-x_{b_3}))\\
\rotatebox[origin=c]{-90}{$\cdots$} && \rotatebox[origin=c]{-90}{$\cdots$}\\
Q^{r_{b_c}, sing}_{d_{b_c}} \subseteq O^{c, r_{b_1}} \subseteq O^{c-1, r_{b_1}} && (k-b_1+1-x_c)((r_{b_{c-1}}-x_{b_{c-1}})-(r_{b_c}-x_{b_c}))\\
O^{c+1, n_{a_t}} \subseteq O^{c+1, r_{b_1}} \subseteq O^{c, r_{b_1}} && (k-b_1+1-a_t)((r_{b_c}-x_{b_c})-(r_{b_{c+1}}-x_{b_{c+1}}))\\
\rotatebox[origin=c]{-90}{$\cdots$} && \rotatebox[origin=c]{-90}{$\cdots$}\\
O^{u, n_{a_t}} \subseteq O^{u, r_{b_1}} \subseteq O^{u-1, r_{b_1}} && (k-b_1+1-a_t)((r_{b_{u-1}}-x_{b_{u-1}})-(r_{b_u}-x_{b_u}))\\
T^t \subseteq T^{t+1} \subseteq O^{u, r_{b_1}} && (k-b_1+1-a_t)(r_{b_u}-x_{b_u})
\end{array}
\end{equation*}
\end{table}

\item
As another case for the row that corresponds to $Q^{r_{b_1}}_{d_{b_1}}$, if $r_{b_1} < n_{a_t}$, then $Z^1$ is determined by $O^{1, n_{a_t}} \subseteq Z^1 \subseteq Q^{r_{b_1}}_{d_{b_1}}$. The linear space $Z^1$ has to be contained in the orthogonal complement of $O^{1, n_{a_t}}$, so $Z^1 \subseteq Q^{r_{b_1}+(a_t-x_{b_1})}_{d_{b_1}-(a_t-x_{b_1})}$. Hence $Z^1$ can be parameterized by $OG(\beta_1, d_{b_1}-r_{b_1}-2(a_t-x_{b_1}))$. If $OG(\beta_1, d_{b_1}-r_{b_1}-2(a_t-x_{b_1}))$ has two components, then $Z^1$ belongs to the component determined by the marking $m_\bullet$. The parameterizations of $T^{t+1}, O^{u, r_{b_1}}, \ldots, O^{2, r_{b_1}}$ are similar to the previous case, the total dimension is $\beta_1(d_{b_1}+x_{b_1}-2(k-b_1+1)-\frac{\beta_1-1}{2})$ as before. The diagram and the parameterizations in this case are as in Figure \ref{figure8} and Table \ref{table3}.
\begin{figure} \caption{Row of Type 4 in $\widetilde{V}$} \label{figure8}
\begin{equation*}
\input{figure8}
\end{equation*}
\end{figure}

\begin{table} \caption{Dimensions for a row of Type 4 in $\widetilde{V}$} \label{table3}
\begin{equation*}
\begin{array}{lll}
\mbox{\emph{Coordinates of $\widetilde{V}$ in the $(t+1)$-st row:}} &\quad& \mbox{\emph{Dimensions of the corresponding Grassmannian:}}\\
O^{1, n_{a_t}} \subseteq Z^1 \subseteq Q^{r_{b_1}}_{d_{b_1}} && \beta_1(d_{b_1}+x_{b_1}-2(k-b_1+1) -(r_{b_1}-x_{b_1})+\frac{\beta_1-1}{2})\\
O^{2, n_{a_t}} \subseteq O^{2, r_{b_1}} \subseteq Z^1 && (k-b_1+1-x_2)((r_{b_1}-x_{b_1})-(r_{b_2}-x_{b_2}))\\
\rotatebox[origin=c]{-90}{$\cdots$} && \rotatebox[origin=c]{-90}{$\cdots$}\\
T^t \subseteq T^{t+1} \subseteq O^{u, r_{b_1}} && (k-b_1+1-a_t)(r_{b_u}-x_{b_u})
\end{array}
\end{equation*}
\end{table}

\item
Finally, the $(t+h)$-th row for some $h \geq 2$ is  similar to the case above. The parameterizations are given by a tower of Grasmanninans contained in an orthogonal Grassmannian and the total dimension adds up to $d_{b_h}+x_{b_h}-2(k-b_h+1)-\frac{\beta_h-1}{2}$. The diagram and the parameterizations are as in Figure \ref{figure9} and Table \ref{table4}.
\begin{figure} \caption{Row of Type 5 in $\widetilde{V}$} \label{figure9}
\begin{equation*}
\input{figure9}
\end{equation*}
\end{figure}

\begin{table} \caption{Dimensions for a row of Type 5 in $\widetilde{V}$} \label{table4}
\begin{equation*}
\begin{array}{lll}
\mbox{\emph{Coordinates of $\widetilde{V}$ in the $(t+h)$-th row:}} &\quad\quad& \mbox{\emph{Dimension of the corresponding Grassmannian:}}\\
O^{h, r_{b_{h-1}}} \subseteq Z^h \subseteq Q^{r_{b_h}} && \beta_h(d_{b_h}+x_{b_h}-2(k-b_h+1) -(r_{b_h}-x_{b_h})+\frac{\beta_h-1}{2})\\
O^{h+1, r_{b_{h-1}}} \subseteq O^{h+1, r_{b_h}} \subseteq Z^h && (b_{h-1}-b_h)((r_{b_h}-x_{b_h})-(r_{b_{h+1}}-x_{b_{h+1}}))\\
\rotatebox[origin=c]{-90}{$\cdots$} && \rotatebox[origin=c]{-90}{$\cdots$}\\
T^{t+h-1} \subseteq T^{t+h} \subseteq O^{u, r_{b_h}} && (b_{h-1}-b_h)(r_{b_u}-x_{b_u})
\end{array}
\end{equation*}
\end{table}
\end{enumerate}

\medskip

The variety $\widetilde{V}$ is smooth as it is an iterated tower of the ordinary and the orthogonal Grassmannian bundles observed above. The inverse image $\pi^{-1}(\Lambda)$ of a point $\Lambda$ in $V$ is irreducible by the same observations, hence $\widetilde{V}$ is irreducible for a marked restriction variety. Furthermore, combining the results from each row of the diagram, $\dim \widetilde{V}$ is given by
\[ \dim \widetilde{V}= \sum_{g=1}^t \alpha_g \left( n_{a_g} - a_g \right) + \sum_{h=1}^u \beta_h\left(d_{b_h}+x_{b_h}-2(k-b_h+1)-\frac{\beta_h-1}{2}\right) =\dim V \]
which concludes the proof.
\end{proof}

\bigskip


Over $V^0(L_\bullet, Q_\bullet)$, the inverse image of a point $\pi^{-1}(\Lambda)$ is determined uniquely by
\begin{align*}
&  T^g=\Lambda \cap L_{n_{a_g}}, \quad T^{t+h}=\Lambda \cap Q^{r_{b_h}}_{d_{b_h}}, \\
&  O^{h, r_{b_\theta}}=\overline{Q^{r_{b_h}, sing}_{d_{b_h}}, \; \Lambda \cap Q^{r_{b_\theta}}_{d_{b_\theta}}}, \quad O^{h, n_{a_\theta}}=\overline{Q^{r_{b_h}, sing}_{d_{b_h}}, \; \Lambda \cap L_{n_{a_\theta}}} \quad \mbox{and}\\
&  Z^h=\overline{Q^{r_{b_h}, sing}_{d_{b_h}}, \; \Lambda \cap Q^{r_{b_h}}_{d_{b_h}}} \quad \quad \mbox{for all} \;\; 1\leq g \leq t, \;\; 1 \leq h \leq u .
\end{align*}
$V^0(L_\bullet, Q_\bullet)$ is in the smooth locus of $V(L_\bullet, Q_\bullet)$ since it is homogeneous under the action of $SO(n)$. Then, Zariski's main theorem shows that $\pi$ is an isomorphism over $V^0(L_\bullet, Q_\bullet)$. Therefore we have

\medskip

\begin{thm}\label{resthm}
The map $\pi : \widetilde{V}(L_\bullet, Q_\bullet) \to V(L_\bullet, Q_\bullet)$ is a resolution of singularities.
\end{thm}

\bigskip


\section{The Exceptional Locus}

We now study the exceptional locus of $\pi$. More specifically, we are interested in the codimension of the components of the exceptional locus.

Corresponding to the three types of conditions in Definition \ref{restdef}, namely, 
\[ \dim(\Lambda \cap Q^{r_i, sing}_{d_i})=x_i \; , \;\; \dim(\Lambda \cap L_{n_j})=j \; , \;\; \mbox{and} \;\; \dim(\Lambda \cap Q^{r_i}_{d_i})=k-i+1 \; , \]
we consider three types of loci $\Sigma$ where $\pi$ has positive dimensional fibers. The image of the exceptional locus of $\pi$ is equal to the union of the following $\Sigma$'s

\begin{description}[leftmargin=*]
\setlength\itemsep{0.5em}
\vspace{0.5em}

\item[I] $\Sigma_{r_{b_h}}$ : The Zariski closure of the subvariety of $V(L_\bullet, Q_\bullet, m_\bullet)$ parameterizing $k$-dimensional isotropic subspaces $\Lambda$ such that $\dim(\Lambda \cap Q^{r_{b_h}, sing}_{d_{b_h}}) = x_{b_h}+1$ for some $1 \leq h \leq u$.

\item[II] $\Sigma_{n_{a_g}}$ : The Zariski closure of the subvariety of $V(L_\bullet, Q_\bullet, m_\bullet)$ parameterizing $k$-dimensional isotropic subspaces $\Lambda$ such that $\dim(\Lambda \cap L_{n_{a_g}}) = a_g +1$ for some $1 \leq g \leq t$, or $\dim(\Lambda \cap L_{n_s}) = s +2$ in a certain case that is discussed in Remark \ref{contained}.

\item[III] $\Sigma_{d_{b_h}}$ : The Zariski closure of the subvariety of $V(L_\bullet, Q_\bullet, m_\bullet)$ parameterizing $k$-dimensional isotropic subspaces $\Lambda$ such that $\dim(\Lambda \cap Q^{r_{b_h}}_{d_{b_h}}) = k-b_h +2$ for some $1 \leq h \leq u-1$.

\vspace{0.5em}
\end{description}

\begin{remark} \label{contained}

Note that these loci do not exist for every restriction variety. There are natural restrictions for their existence resulting from the properties of quadrics. The locus $\Sigma_{n_{a_g}}$, for some $1 \leq g \leq t$, exists only if $n_{a_g}>a_g$, and the locus $\Sigma_{d_{b_h}}$, for some $1 \leq h \leq u-1$, exists only if $u>1$ and $d_{b_h}-r_{b_h}-2\beta_1 \geq 3$ (requirement for the irreducibility of the quadric that arises in the sequence of $\Sigma_{d_{b_h}}$). Similarly, the locus $\Sigma_{r_{b_h}}$, for some $1 \leq h \leq u$, exists only if $r_{b_h}>x_{b_h}$

Furthermore, different components of $OG(m, 2m)$ must be kept in mind in a certain case. Suppose $b_1$ is a special index (that is, $x_{b_1} = k-b_1+1 - \frac{d_{b_1} - r_{b_1}}{2}$ as in Definition \ref{def3}) and $2n_s = r_{b_1} + d_{b_1}$. The linear space $L_{n_s}$ belongs to one of the components of the Fano variety of maximal dimensional linear spaces contained in $Q^{r_{b_1}}_{d_{b_1}}$. For a general $k$-plane $\Lambda$, the $n_s$-dimensional linear subspace $\Lambda \cap Q^{r_{b_1}}_{d_{b_1}}$ lies in the other component. Note that two linear spaces in $OG(m, 2m)$ belong to the same component if and only if their intersection is equal to $m$ mod 2. Therefore, no $\Lambda$ in $V$ satisfies $\dim(\Lambda \cap L_{n_s}) = s +1$, but there may be elements with $\dim(\Lambda \cap L_{n_s}) = s +2$.

\end{remark}
\smallskip

\begin{ex}
The locus $\Sigma_{r_{b_1}}$ does not make sense for the restriction variety given by $\Big[ Q^0_8 \subseteq Q^0_9 \Big]$ since $Q^{0, sing}_9$ is empty. Similarly the locus $\Sigma_{r_{b_1}}$ does not exist for the restriction variety given by $\Big[ L_1 \subseteq Q^1_7 \Big]$ since $x_1=1$ and it is not possible to intersect $Q^{1, sing}_7$ in a higher dimension.
\end{ex}

\begin{ex}
The loci $\Sigma_{n_{a_g}}$ do not exist for the restriction variety given by $\Big[ L_1 \subseteq L_7 \subseteq L_8 \Big]$; lines contained in $L_8$ containing $L_1$ cannot intersect $L_8$ or $L_1$ in higher dimensions. Similarly, $\Sigma_{d_{b_1}}$ does not exist for the restriction variety given by $\Big[ Q^2_7 \subseteq Q^1_8 \Big]$.
\end{ex}

\begin{ex}
Let $V=\Big[ L_2 \subseteq Q^0_4 \Big]$, the variety of lines contained in a smooth quadric surface intersecting a fixed line on the surface. This (marked) restriction variety is one of the components of the lines on the quadric surface. The locus $\Sigma_{n_{a_1}}$, given by $\Big[ L_1 \subseteq L_2 \Big]=L_2$, lies in the other component, and hence is not contained in $V$. It is easy to see $V$ is smooth in this example as it is isomorphic to $\mathbb{P}^1$.
\end{ex}

\begin{ex}
Let $V$ be the restriction variety in $OG(4, 8)$ given by $\Big[ L_1 \subseteq L_3 \subseteq L_4 \subseteq Q^1_7 \Big]$. A general element $\Lambda$ of $V$ satisfies $\dim (\Lambda \cap L_4) = 3$, therefore $L_4$ and $\Lambda$ lie in different components of $OG(4, 8)$. This shows that the restriction variety given by the sequence $\Big[ L_1 \subseteq L_2 \subseteq L_3 \subseteq L_4 \Big]=L_4$ is not in the image of the exceptional locus of $\pi$ in this case.
\end{ex}

\begin{ex}
Let $V$ be given by $\Big[ L_3 \subseteq L_4 \subseteq Q^1_7 \subseteq Q^0_8 \Big]$. A general element $\Lambda$ of $V$ satisfies $\dim (\Lambda \cap L_4) = 2$, and hence $L_4$ lies in the same component of $OG(k, 2k)$ as $V$. Since we have $\dim (\Lambda \cap L_4) = 4$ mod 2 for linear spaces $\Lambda$ in the same component as $L_4$, we conclude $\dim (\Lambda \cap L_4)$ must be either 2 or 4. Therefore, in this case we have $\Sigma_{n_{a_1}} = \Big[ L_1 \subseteq L_2 \subseteq L_3 \subseteq L_4 \Big]$.
\end{ex}

\bigskip

In the following, we study loci $\Sigma$ that are contained in the restriction variety $V$. Over each $\Sigma$,  $\pi^{-1} (\Sigma)$ is irreducible of codimension
\[ \codim (\pi^{-1} (\Sigma))= \codim(\Sigma) - \dim(\pi^{-1}(\Lambda)) \]
for a general point $\Lambda$ in $\Sigma$. We now consider each $\Sigma$ separately, observe the sequence that defines the restriction variety $\Sigma$, and study $\codim (\pi^{-1} (\Sigma))$ in each case. Our aim is to find the components of the exceptional locus with codimension greater than 1.

\begin{obs} \label{codimcomp}
A component of the exceptional locus of $\pi$ with image of one of the types
\begin{itemize}
\item $\Sigma_{r_{b_h}}$ with $r_{b_h} < n_s$
\item $\Sigma_{n_{a_g}}$ with $1 \leq g \leq t-1$
\item $\Sigma_{d_{b_h}}$ for all $1 \leq h \leq u-1$ 
\end{itemize}
has codimension larger than 1 (by I.B, I.C, II.B and III below). A component with image of type $\Sigma_{r_{b_h}}$ with $r_{b_h} \geq n_s$ has codimension equal to 1 (by I.A and I.D below). A component with image of type $\Sigma_{n_s}$ has codimension given by $\codim (\pi^{-1}(\Sigma_{n_s}))  = d_{k-s} + x_{k-s} -s - n_s  - 1$ which may be larger than or equal to 1. 
\end{obs}

In the following computation which results in Observation \ref{codimcomp}, each component of the exceptional locus is studied by dividing it into subcases.
\bigskip


\begin{description}[leftmargin=1.5em]

\item[I]
$\Sigma_{r_{b_h}}$ : $\dim(\Lambda \cap Q^{r_{b_h}, sing}_{d_{b_h}}) = x_{b_h}+1$ for some $1 \leq h \leq u$ \\

Given the corank $r_{b_h}$, we divide this case into sub-cases depending on the relation between $r_{b_h}$ and the dimensions $n_{a_g}$ of the isotropic linear spaces appearing in the sequence defining $V$. The sub-cases we consider in the following are:
\begin{description}
\item[I.A] $r_{b_h} > n_s$
\item[I.B] $r_{b_h}<n_s$ and $r_{b_h} \neq n_j$ for all $j$
\item[I.C] $r_{b_h}=n_j$ for some $n_j<n_s$
\item[I.D] $r_{b_h}=n_s$
\end{description}
\bigskip

\item[I.A]
Suppose $r_{b_h}>n_s$. A general element of $\Sigma_{r_{b_h}}$ intersects $Q^{r_{b_h}, sing}_{d_{b_h}}$ in one more dimension. Equivalently, this is the restriction variety associated to the sequence obtained by replacing the sub-quadric $Q^{r_{k-s}}_{d_{k-s}}$ with the isotropic linear space $L_{r_{b_h}}$ in the fixed full flag of dimension $r_{b_h}$. Note that $\Sigma_{r_{b_{h-1}}}$ contains $\Sigma_{r_{b_h}}$, so all $\Sigma_{r_{b_h}}$ with $r_{b_h}>n_s$ are contained in $\Sigma_{r_{b_1}}$. Therefore it is sufficient to consider $\Sigma_{r_{b_1}}$.

\begin{ex}
Let $V$ be the restriction variety given by the sequence $\Big[ L_3 \subseteq Q^7_{10} \subseteq Q^5_{20} \Big]$. The loci $\Sigma_{r_{b_1}}$ and $\Sigma_{r_{b_2}}$ are defined as the closures of the loci:
\[ \Sigma^0_{r_{b_1}} := \left\{ \Lambda \in V \; \big| \; \dim (\Lambda \cap Q^{7, sing}_{10})=2 \; \mbox{with other conditions of $V^0$ unchanged} \right\} \]
\[ \Sigma^0_{r_{b_2}} := \left\{ \Lambda \in V \; \big| \; \dim (\Lambda \cap Q^{5, sing}_{20})=2 \; \mbox{with other conditions of $V^0$ unchanged} \right\} \]
Since $\Sigma_{r_{b_2}}$ is contained in $\Sigma_{r_{b_1}}=\Big[ L_3 \subseteq L_7 \subseteq Q^5_{20} \Big]$, it is sufficient to consider $\codim (\pi^{-1} (\Sigma_{r_{b_1}}) )$.
\end{ex}
\smallskip

The introduced isotropic linear space $L_{r_{b_h}}$ is contained in $Q^{r_i}_{d_i}$ for $b_1 \leq i < k-s$. Therefore, in the resulting restriction variety, the value of $x_i$ increases by one for $b_1 \leq i < k-s$. Thus we have
\begin{align*}
\codim (\Sigma_{r_{b_1}})  & = \big( d_{k-s} + x_{k-s} -2(s+1) \big)  - \big( (r_{b_1}-(s+1))-(\beta_1-1) \big)  \\
& = d_{k-s} -r_{b_1} -\beta_1
\end{align*}
since $x_{k-s}=s$ by our assumption that $r_{b_1}>n_s$.
\smallskip

Now we study the inverse image $\pi^{-1}(\Lambda)$ of a general point $\Lambda$ in $\Sigma_{r_{b_1}}$. By assumption there is no $O$ containing $Q^{r_{b_1}, sing}_{d_{b_1}}$ and $O$'s contained in $Q^{r_{b_1}, sing}_{d_{b_1}}$ are determined uniquely by $\Lambda$. We have $\overline{T^{t+1}, Q^{r_{b_1}, sing}_{d_{b_1}}} \subseteq Z_1 \subseteq Q^{r_{b_1}}_{d_{b_1}}$ where $\dim(\overline{T^{t+1}, Q^{r_{b_1}, sing}_{d_{b_1}}})=r_{b_1}+(k-b_1+1)-(x_{b_1}+1)$ and $\dim Z_1=\dim(\overline{T^{t+1}, Q^{r_{b_1}, sing}_{d_{b_1}}})+1$. Since $Z^1$ has to lie in the orthogonal complement of $\;\overline{T^{t+1}, Q^{r_{b_1}, sing}_{d_{b_1}}}\;$, we have $\;\overline{T^{t+1}, Q^{r_{b_1}, sing}_{d_{b_1}}} \subseteq Z_1 \subseteq Q^{r_{b_1}+(k-b_1+1-x_{b_1}-1)}_{d_{b_1}-(k-b_1+1-x_{b_1}-1)}\;$. Such $Z_1$ can be parameterized by $OG(1, d_{b_1}-r_{b_1}-2(k-b_1+1-x_{b_1}-1))$. Therefore
\begin{align*}
\codim(\pi^{-1}(\Sigma_{r_{b_1}})) & =d_{k-s} - r_{b_1} -\beta_1 - \Big( d_{b_1}-r_{b_1}-2(k-b_1+1-x_{b_1}-1) -2 \Big) \\
& = d_{k-s} - d_{b_1} +2(k-b_1+1-x_{b_1}) -\beta_1 \\
& =1
\end{align*}
since $d_{b_1}-d_{k-s}=\beta_1-1$ and $k-b_1+1-s=\beta_1$.

\begin{ex} Let $V=\Big[ L_3 \subseteq Q^7_{10} \subseteq Q^5_{20} \Big]$, then 
\begin{align*}
\widetilde{V} = \{ (T^1, T^2, T^3, Z^1, Z^2, O^{2, r_{b_1}}) \; | \; & Q^{5, sing}_{20} \subseteq O^{2, r_{b_1}} \subseteq Z^2, \;\; Q^{7, sing}_{10} \subseteq Z^1, \\
& T^1 \subseteq L_3, \;\; T_2 \subseteq O^{2, r_{b_1}} \subseteq Z^1 \subseteq Q^7_{10}, \;\; T^3 \subseteq Z^2 \subseteq Q^5_{20} \; , \}
\end{align*}
equivalently, the diagram is given in Figure \ref{figure10}.
\begin{figure} \caption{$\widetilde{V}$ for $\Big[ L_3 \subseteq Q^7_{10} \subseteq Q^5_{20} \Big]$} \label{figure10}
\begin{equation*}
\input{figure10}
\end{equation*}
\end{figure}

The subvariety $\Sigma_{r_{b_1}}=\Big[ L_3 \subseteq L_7 \subseteq Q^5_{20} \Big] \subseteq V$ has codimension 2. In the inverse image $\pi^{-1}(\Lambda)$ of a general point $\Lambda$ in $\Sigma_{r_{b_1}}$, we have ${T^3=\Lambda}, \;\; {T^2=\Lambda\cap Q^7_{10}=\Lambda\cap L_7}, \;\; {T^1=\Lambda \cap L_3}, \;\; \newline {O^{2, r_{b_1}}=\overline{Q^{5, sing}_{20}, \Lambda\cap Q^7_{10}}}$, ${Z^2=\overline{Q^{5, sing}_{20}, \Lambda}}$ and ${Q^{7, sing}_{10} \; \subseteq \; Z^1 \; \subseteq \; Q^7_{20}}$ where $\dim Z^1=8$. The linear space $Z^1$ is parameterized by a smooth plane quadric, or equivalently, $OG(1, 3)$. Thus $\dim(\pi^{-1}(\Lambda))=1$ and $\codim (\pi^{-1}(\Sigma_{r_{b_1}}))=2-1=1$.
\end{ex}
\medskip

\begin{ex}
Let $V= \Big[ L_1 \subseteq Q^3_6 \subseteq Q^1_8 \Big]$, an orthogonal Schubert variety in $OG(3, 9)$. The diagram in Figure \ref{figure11} defines $\widetilde{V}$.
\begin{figure} \caption{$\widetilde{V}$ for $\Big[ L_1 \subseteq Q^3_6 \subseteq Q^1_8 \Big]$} \label{figure11}
\begin{equation*}
\input{figure11}
\end{equation*}
\end{figure}

The subvariety $\Sigma_{r_{b_1}} = \Big[ L_1 \subseteq L_3 \subseteq Q^1_8 \Big]$ has codimension 2. In the inverse image $\pi^{-1}(\Lambda)$ of a general point $\Lambda$ in $\Sigma_{r_{b_1}}$, only $Z^1$ is not determined uniquely. We have $\dim Z^1=4$ and $Q^{3, sing}_6 \subseteq Z^1 \subseteq Q^3_6$, from which we conclude $Z^1$ is parameterized by $OG(1, 3)$. Thus $\dim(\pi^{-1}(\Lambda))=1$ and $\codim (\pi^{-1}(\Sigma_{r_{b_1}}))=2-1=1$.
\end{ex}
\medskip

\item[I.B] Next we consider $\Sigma_{r_{b_h}}$ such that there are $L_{n_j}$ in the sequence with $r_{b_h}<n_j$ but no $L_{n_j}$ with $n_j=r_{b_h}$. Let $n_{j_\sharp}:= \min\{n_j \; | \; r_{b_h}<n_j \}$. If $r_{b_{h-1}}$ satisfies  $r_{b_h} < r_{b_{h-1}}<n_{j_\sharp}$ then $\Sigma_{r_{b_{h-1}}}$ contains $\Sigma_{r_{b_h}}$. Therefore it is sufficient to consider $r_{b_h}$ such that $r_{b_h}<n_{j_\sharp}<r_{b_{h-1}}$.
\smallskip

For a general element $\Lambda \in \Sigma_{r_{b_h}}$, and a general element $W \in V$, we have $\dim (\Lambda \cap Q^{r_{b_h}}_{d_{b_h}}) = \dim (W \cap Q^{r_{b_h}}_{d_{b_h}}) +1$, and $\dim (\Lambda \cap L_{n_{j_\sharp}}) = \dim (W \cap L_{n_{j_\sharp}})$. Therefore, in the sequence of $\Sigma_{r_{b_h}}$, the isotropic linear space $L_{n_{j_\sharp}}$ is replaced with $L_{r_{b_h}}$, and the sub-quadric $Q^{r_{i_0}}_{d_{i_0}}$, where $r_{i_0}:= \max \{ r_i \leq n_{j_\sharp} \}$, is replaced with $Q^{n_{j_\sharp}}_{d_{i_0}-n_{j_\sharp}+r_{i_0}}$.
\smallskip

This scenario arises in the study of other types of components of the exceptional locus. Here we give the general rule that applies whenever an isotropic linear space is replaced with a smaller dimensional isotropic linear space.

\begin{rul} \label{rule1}
Given the defining sequence of a restriction variety, consider the modified sequence where an isotropic linear space $L_{n_j}$ is replaced with a smaller dimensional isotropic linear space. If there are sub-quadrics $Q^{r_i}_{d_i}$ in the sequence satisfying $r_i < n_j$, then let $r_{i_0}:= \max \{r_i < n_j \}$, and replace $Q^{r_{i_0}}_{d_{i_0}}$ with $Q^{n_j}_{d_{i_0} - (n_j - r_{i_0})}$.
\end{rul}

Since $L_{r_{b_h}}$ is contained in the singular locus of every sub-quadric in the group of $Q^{r_{b_h}}_{d_{b_h}}$, for each of these sub-quadrics, $x_i$ increases by one. Hence we get
\[ \codim (\Sigma_{r_{b_h}})=n_{j_\sharp}-r_{b_h} + n_{j_\sharp}-r_{i_0} - \beta_h . \]

\begin{ex} Let $V=\Big[ L_7 \subseteq Q^5_{15} \subseteq Q^2_{25} \Big]$, then  $\Sigma_{r_{b_1}} = \Big[ L_5 \subseteq Q^7_{13} \subseteq Q^2_{25} \Big]$. Specializing a general element $\Lambda$ of $V$ so that it intersects $L_5$ increases $x_2$ by 1. In this example, $\codim (\Sigma_{r_{b_1}}) =2+2-1= 3$. \end{ex}
\smallskip

Note that the linear space $L_{n_{j_\sharp}}$ may not be among $L_{n_{a_g}}$, that is, the largest dimensional isotropic linear space in a group with consecutively increasing dimensions. Let $L_{n_{a_{g_\sharp}}}$ be the smallest $L_{n_{a_g}}$ containing $L_{n_{j_\sharp}}$. In the inverse image $\pi^{-1}(\Lambda)$ of a general point $\Lambda$ in $\Sigma_{r_{b_h}}$, all coordinates are determined uniquely except for $O^{h, n_{a_{g_\sharp}}}$ and $Z^h$. We have $\overline{Q^{r_{b_h}, sing}_{d_{b_h}}, \Lambda\cap L_{n_{a_{g_\sharp}}}} \subseteq O^{h, n_{a_{g_\sharp}}} \subseteq L_{n_{a_{g_\sharp}}}$ thus $O^{h, n_{a_{g_\sharp}}}$ can be parameterized by $G(1, n_{a_{g_\sharp}}- (r_{b_h}+a_{g_\sharp}-x_{b_h})+1)$. Then $Z^h$ is determined uniquely as $\overline{O^{h, n_{a_{g_\sharp}}}, \Lambda \cap Q^{r_{b_h}}_{d_{b_h}}}$. Thus ${\dim (\pi^{-1}(\Lambda))=n_{a_{g_\sharp}}- (r_{b_h}+a_{g_\sharp}-x_{b_h})}$ and
\begin{align*}
\quad \quad \quad \codim (\pi^{-1}(\Sigma_{r_{b_h}})) &= n_{j_\sharp}-r_{b_h} + n_{j_\sharp}-r_{b_{h-1}+1}-\beta_h- (n_{a_{g_\sharp}}- (r_{b_h}+a_{g_\sharp}-x_{b_h}))\\
& = \Big( (n_{j_\sharp}-r_{b_h}) - (n_{a_{g_\sharp}}- (r_{b_h}+a_{g_\sharp}-x_{b_h})) \Big) + n_{j_\sharp} - r_{b_{h-1}+1}- \beta_h\\
& \geq 2
\end{align*}
since $\Big( (n_{j_\sharp}-r_{b_h}) - (n_{a_{g_\sharp}}- (r_{b_h}+a_{g_\sharp}-x_{b_h})) \Big) \geq 1$ and $n_{j_\sharp} - r_{b_{h-1}+1}- \beta_h \geq 1$ by construction.

\begin{ex} Let $V=\Big[ L_6 \subseteq L_7 \subseteq Q^2_{15} \Big]$, then $\widetilde{V}$ is given by the diagram in Figure \ref{figure12}.
\begin{figure} \caption{$\widetilde{V}$ for $\Big[ L_6 \subseteq L_7 \subseteq Q^2_{15} \Big]$} \label{figure12}
\begin{equation*}
\input{figure12}
\end{equation*}
\end{figure}
The subvariety $\Sigma_{r_{b_1}} = \Big[ L_2 \subseteq L_7 \subseteq Q^2_{15} \Big]$ has codimension 7. In the inverse image $\pi^{-1}(\Lambda)$ of a general point $\Lambda$ in $\Sigma_{r_{b_1}}$, we have $T^2=\Lambda, T^1=\Lambda \cap L_7$. As above, $Z^1$ is determined as $Z^1=\overline{ O^{1, n_{a_1}}, \Lambda}$ so the nontrivial part is the parametrization of $O^{1, n_{a_1}}$. We have $\overline{ Q^{2, sing}_{15}, \Lambda \cap L_7} \subseteq O^{1, n_{a_1}} \subseteq L_7$ which is parameterized by $G(1, 4)$. Thus $\dim(\pi^{-1}(\Lambda))=3$ and $\codim\pi^{-1}(\Sigma_{r_{b_1}})=7-3=4$.
\end{ex}
\medskip

\begin{ex} Let $V=\Big[ L_7 \subseteq Q^5_{15} \subseteq Q^2_{25} \Big]$, then $\widetilde{V}$ is given by the diagram in Figure \ref{figure13}.
\begin{figure} \caption{$\widetilde{V}$ for $\Big[ L_7 \subseteq Q^5_{15} \subseteq Q^2_{25} \Big]$} \label{figure13}
\begin{equation*}
\input{figure13}
\end{equation*}
\end{figure}
The subvariety $\Sigma_{r_{b_1}} = {\Big[ L_5 \subseteq Q^7_{13} \subseteq Q^2_{25} \Big]}$ has codimension 3. In the inverse image $\pi^{-1}(\Lambda)$ of a general point $\Lambda$ in $\Sigma_{r_{b_1}}$, we have $T^3=\Lambda$,  $\;T^2 =\Lambda \cap Q^5_{15}= \Lambda \cap Q^7_{13}$, $\;T^1=\Lambda \cap L_7=\Lambda \cap L_5$, $\;O^{2, n_{a_1}}=\overline{Q^{2, sing}_{25}, \Lambda \cap L_7}$, $\;O^{2, r_{b_1}}={\overline{Q^{2, sing}_{25}, \Lambda \cap Q^5_{15}}}$, $\;Z^2={\overline{Q^{2, sing}_{25}, \Lambda}}$. The linear space $O^{1, n_{a_1}}$ satisfies ${Q^{5, sing}_{15} \; \subseteq \; Q^5_{15} \; \subseteq \; L_7}$ and hence can be parameterized by $G(1, 2)$. Then $Z^1$ is determined uniquely as $Z^1={\overline{ O^{1, n_{a_1}}, \Lambda \cap Q^5_{15}}}$. Thus $\dim(\pi^{-1}(\Lambda))=1$ and $\codim (\pi^{-1}(\Sigma_{r_{b_1}}))=3-1=2$.
\end{ex}

\medskip

\item[I.C] Consider $r_{b_h}$ such that there is $L_{n_j}$ with $n_j < s$ in the defining sequence satisfying $n_j = r_{b_h}$. Since there is no $L_{n_j}$ in the sequence with $n_j=r_{b_h}+1$, we have $r_{b_h}=n_{a_g}$ for some $1 \leq g \leq t-1$.

In the sequence of $\Sigma_{r_{b_h}}$, the isotropic linear space that appears next to $L_{n_{a_g}}$ in the sequence of $V$, namely $L_{n_{(a_g+1)}}$, is replaced with $L_{n_{a_g}}$. Consequently, $L_{n_{a_g}}$ is replaced with the isotropic linear space of one less dimension, namely $L_{n_{a_g}-1}$. Similarly, each isotropic linear space $L_\tau$, where $n_{a_g} - \alpha_g +1 \leq \tau \leq n_{a_g}$, is replaced with $L_{\tau-1}$. Applying Rule \ref{rule1}, the sub-quadric $Q^{r_{i_0}}_{d_{i_0}}$, where $r_{i_0}:= \max \{ r_i \leq n_{a_g+1} \}$, is replaced with $Q^{n_{a_g+1}}_{d_{i_0} - (n_{a_g+1} - r_{i_0})}$. Observe that $x_i$ increases by one for each $i$ satisfying $b_h + \beta_h -1 \geq i \geq b_h$.
\smallskip

Comparing the dimensions of both sequences, we have
\begin{align*}
\codim (\Sigma_{r_{b_1}}) & = \alpha_{g} \left(n_{a_{g}} - a_g \right) +  \alpha_{g +1} \left( n_{a_{g+1}} - a_{g+1} \right) \\
& \quad - (\alpha_g+1) \left( n_{a_g} - a_g -1 \right) - (\alpha_{g+1}-1) \left(n_{a_{g+1}} - a_{g+1} \right) \\
& \quad + (n_{a_g+1}-r_{i_0})-\beta_h .
\end{align*}

In the inverse image $\pi^{-1}(\Lambda)$ of a general point $\Lambda$ in $\Sigma_{r_{b_h}}$, all coordinates are determined uniquely except for $O^{h, n_{a_{g+1}}}$, $Z^h$ and the coordinates in the $g$-th row. We have \newline ${\overline{Q^{r_{b_h}, sing}_{d_{b_h}}, \Lambda\cap L_{n_{a_{g+1}}}} \subseteq O^{h, n_{a_{g+1}}} \subseteq L_{n_{a_{g+1}}}}$ thus $O^{h, n_{a_{g+1}}}$ can be parameterized by \newline ${G(1, n_{a_{g+1}}- (r_{b_h}+a_{g+1}-x_{b_h})+1)}={G(1, n_{a_{g+1}}-n_{a_g}-\alpha_{g+1}+1)}$. Then $Z^h$ is determined uniquely as  ${\overline{O^{h, n_{a_{g+1}}}, \Lambda \cap Q^{r_{b_h}}_{d_{b_h}}}}$. On the other hand, the $g$-th row is determined uniquely once $T^g$ is determined. The linear space $T^g$ satisfies  ${T^{g-1} \subseteq T^g \subseteq \Lambda \cap L_{n_{a_g}}}$ and hence can be parameterized by $G(\alpha_g, \alpha_{g}+1)$. Thus $\dim (\pi^{-1}(\Lambda))={n_{a_{g+1}}-n_{a_g}-\alpha_{g+1} + \alpha_g}$ and
\[ \codim (\pi^{-1}(\Sigma_{r_{b_h}}))=n_{a_g+1}-r_{i_0}- \beta_h +1 \]
which is greater than 1, as $n_{a_g+1}-r_{i_0}- \beta_h \geq 1$ by construction.

\begin{ex}
Let $V=\Big[ L_2 \subseteq L_4 \subseteq Q^2_7 \big]$, an orthogonal Schubert variety in $OG(3, 9)$. The definition of $\widetilde{V}$ is given by the diagram in Figure \ref{figure14}.
\begin{figure} \caption{$\widetilde{V}$ for $\Big[ L_2 \subseteq L_4 \subseteq Q^2_7 \Big]$} \label{figure14}
\begin{equation*}
\input{figure14}
\end{equation*}
\end{figure}

The subvariety $\Sigma_{r_{b_1}}$ is given by the sequence $\Big[ L_1 \subseteq L_2 \subseteq L_4 \Big]$ as $Q^2_7$ becomes $L_4$ if its corank is increased by 2. The variety $\Sigma_{r_{b_1}}$ has codimension 4. In the inverse image $\pi^{-1} (\Lambda)$ of a general point $\Lambda$ in $\Sigma_{r_{b_1}}$, the coordinates $T^3$ and $O^1$ are determined uniquely as $T^3=O^1=\Lambda$ and $T^2=L_2$. The coordinate $T^1$ satisfies $T^1 \subseteq L_2$ and is parameterized by $G(1, 2)$. The coordinate $Z^1$ satisfies $O^1 \subseteq Z^1 \subseteq Q^3_7$ and is parameterized by $OG(1, 3)$. Thus $\dim(\pi^{-1}(\Lambda))=2$ and $\codim (\pi^{-1}(\Sigma_{r_{b_1}}))=2$.
\end{ex}

\begin{ex} Let $V= \Big[ L_5 \subseteq L_{10} \subseteq Q^6_{19} \subseteq Q^5_{20} \subseteq Q^2_{30} \Big]$, then $\widetilde{V}$ is given by the diagram in Figure \ref{figure15}.
\begin{figure} \caption{$\widetilde{V}$ for $\Big[ L_5 \subseteq L_{10} \subseteq Q^6_{19} \subseteq Q^5_{20} \subseteq Q^2_{30} \Big]$} \label{figure15}
\begin{equation*}
\input{figure15}
\end{equation*}
\end{figure}
The subvariety $\Sigma_{r_{b_1}} = {\Big[ L_4 \subseteq L_5 \subseteq Q^{10}_{15} \subseteq Q^9_{16} \subseteq Q^2_{30} \Big]}$ has codimension 12. In the inverse image $\pi^{-1}(\Lambda)$ of a general point $\Lambda$ in $\Sigma_{r_{b_1}}$, we have $T^4=\Lambda$, $\;T^3=\Lambda \cap Q^9_{16}= \Lambda \cap Q^5_{20}$, $\;T^2 =\Lambda \cap L_5=\Lambda \cap L_{10}$, $\;O^{2, n_{a_2}}={\overline{ Q^{2, sing}_{30}, \Lambda \cap L_{10}}}$, $\;O^{2, r_{b_1}}={\overline{ Q^{2, sing}_{30}, \Lambda \cap Q^5_{20} }}$, $\;Z^2={\overline{ Q^{2, sing}_{30}, \Lambda}}$. The linear space $O^{1, n_{a_2}}$ satisfies ${\overline{ Q^{5, sing}_{20}, \Lambda \cap L_{10}} \; \subseteq \; O^{1, n_{a_2}} \; \subseteq \; L_{10}}$ and hence can be parameterized by $G(1, 5)$. Then $Z^1$ is determined uniquely as $Z^1={\overline{O^{1, n_{a_2}}, \Lambda\cap Q^5_{20}}}$. On the other hand, $T^1$ satisfies ${T^1 \subseteq \Lambda \cap L_5}$ and hence can be parameterized by $G(1, 2)$. Then $ O^{2, n_{a_1}}={\overline{ Q^{2, sing}_{30}, T^1}}$. Thus $\dim(\pi^{-1}(\Lambda))=5$ and $\codim (\pi^{-1}(\Sigma_{r_{b_1}}))=12-5=7$.
\end{ex}

\medskip

\item[I.D] Suppose $r_{b_h}=n_s$. Note that $\Sigma_{r_{b_{h-1}}}$ contains $\Sigma_{r_{b_h}}$, so all $\Sigma_{r_{b_h}}$ are contained in $\Sigma_{r_{b_1}}$. Therefore it is sufficient to consider $\Sigma_{r_{b_1}}$.
\smallskip

In the sequence of $\Sigma_{r_{b_1}}$, the sub-quadric $Q^{r_{k-s}}_{d_{k-s}}$ is replaced with $L_{n_{a_t}}$, and consequently each isotropic linear space $L_\tau$, where $n_{a_t }- \alpha_t +1 \leq \tau \leq n_{a_t}$, is replaced with the isotropic liner space of one less dimension, $L_{\tau-1}$. This increases the value of $x_i$ by one for $i$ satisfying $b_1 \leq i < k-s$. We have
\begin{align*}
\codim\Sigma_{r_{b_1}} & = \alpha_t \left( n_{a_t} - a_t \right) + \sum_{t=1}^{\beta_1} \left( d_{b_1} + x_{b_1} - 2(s+\beta_1) + t-1 \right) \\
& \quad - (\alpha_t + 1) \left( n_{a_t} - a_t -1 \right) - \sum_{t=1}^{\beta_1-1} \left( d_{b_1} + x_{b_1} - 2(s+\beta_1) + t-1 \right) - (\beta_1-1) \\
& = \alpha_t + d_{b_1}  - n_s  - 2\beta_1 +1
\end{align*}

Note that the resulting sequence may contradict condition (9) in Definition \ref{admissible}. The sub-quadric with maximal corank that is smaller than $r_{b_1}$, namely $Q^{r_{b_2}+\beta_2-1}_{d_{b_2}-\beta_2+1}$, may have corank one less than the dimension the introduced isotropic linear space, namely, $L_{n_{a_t}-\alpha_t}$. We remedy this by replacing this sub-quadric with one with larger corank which reflects the geometry of the resulting restriction variety better. Explicitly, if $n_{a_t}-\alpha_t = r_{b_2}+\beta_2$, we replace the sub-quadric $Q^{r_{b_2}+\beta_2-1}_{d_{b_2}-\beta_2+1}$ with $Q^{n_{a_t}}_{d_{b_2}+r_{b_2}-n_{a_t}}$. The changes in the dimension and the value of $x_i$ cancel each other, hence we get the same codimension computation.
\smallskip

This scenario arises in the study of other types of components of the exceptional locus. Here we give the general rule that applies whenever a sub-quadric is replaced with an isotropic linear space.

\begin{rul} \label{rule2}
Given the defining sequence of a restriction variety, consider the modified sequence where a sub-quadric is replaced with an isotropic linear space. If $n_j - r_i =1$ for an isotropic linear space $L_{n_j}$, and a sub-quadric $Q^{r_i}_{d_i}$ in the modified sequence, then let $n_{a_{g_0}}:= \min \{ n_{a_g} \geq n_j \}$, and replace $Q^{r_i}_{d_i}$ with $Q^{n_{a_{g_0}}}_{d_i - (n_{a_{g_0}} - r_i)}$.
\end{rul}

We again look at the fibers of $\pi$. By assumption there is no $O$ containing $Q^{r_{b_1}}_{d_{b_1}}$ and other $O$'s are determined uniquely as there is no change in the relevant rank conditions. The only nontrivial parameterizations are observed for $Z^1$ and the coordinates in the $t$-th row. As in (I.A), we have ${\overline{T^{t+1}, Q^{r_{b_1}, sing}_{d_{b_1}}} \; \subseteq \; Z_1 \; \subseteq \; Q^{r_{b_1}}_{d_{b_1}}}$ where ${\dim(\overline{T^{t+1}, Q^{r_{b_1}, sing}_{d_{b_1}}})}={r_{b_1}+(k-b_1+1)-(x_{b_1}+1)}$ and ${\dim (Z_1)}={\dim(\overline{T^{t+1}, Q^{r_{b_1}, sing}_{d_{b_1}}})+1}$. Since $Z^1$ has to lie in the orthogonal complement of $\;{\overline{T^{t+1}, Q^{r_{b_1}, sing}_{d_{b_1}}}}$, we have  $\;{\overline{T^{t+1}, Q^{r_{b_1}, sing}_{d_{b_1}}} \; \subseteq \; Z_1 \; \subseteq \; Q^{r_{b_1}+(k-b_1+1-x_{b_1}-1)}_{d_{b_1}-(k-b_1+1-x_{b_1}-1)}}$. Such $Z_1$ can be parameterized by  $OG(1, d_{b_1}-r_{b_1}-2(k-b_1+1-x_{b_1}-1))$. On the other hand, the $t$-th row can be determined once $T^1$ is determined. The linear space $T^1$ satisfies ${T^t \subseteq \Lambda \cap L_{n_{a_t}}}$ and hence can be parameterized by $G(\alpha_t, \alpha_t+1)$. Thus $\dim (\pi^{-1}(\Lambda))= {d_{b_1}-r_{b_1}-2(k-b_1+1-x_{b_1}-1)-2+\alpha_t}$ and we have
\[ \codim (\pi^{-1}(\Sigma_{r_{b_1}}))=1 .\]

\begin{ex}
Let $V = \Big[ L_2 \subseteq L_3 \subseteq Q^3_7 \Big]$, an orthogonal Schubert variety in $OG(3, 9)$. The diagram in Figure \ref{figure16} defines $\widetilde{V}$.
\begin{figure} \caption{$\widetilde{V}$ for $ \Big[ L_2 \subseteq L_3 \subseteq Q^3_7 \Big]$} \label{figure16}
\begin{equation*}
\input{figure16}
\end{equation*}
\end{figure}

The subvariety $\Sigma_{r_{b_1}} = \Big[ L_1 \subseteq L_2 \subseteq L_3 \Big]$, which consists of a single point, has codimension 4. In the inverse image $\pi^{-1}(\Lambda)$ of a general point $\Lambda$ in $\Sigma_{r_{b_1}}$, we have $T^1 \subseteq L_3$ which is parameterized by $G(2, 3)$. Also, $\dim (Z^1) = 4$ with $Q^{3, sing}_6 \subseteq Z^1 \subseteq Q^3_6$, so $Z^1$ is parameterized by $OG(1, 3)$ which has dimension 1. Thus $\dim(\pi^{-1}(\Lambda))=3$ and $\codim (\pi^{-1}(\Sigma_{r_{b_1}}))=1$.
\end{ex}
\medskip

\begin{ex} Let $V= \Big[ L_5 \subseteq  Q^5_{10} \subseteq Q^2_{30} \Big]$, then $\widetilde{V}$ is given by the diagram diagram in Figure \ref{figure17}.
\begin{figure} \caption{$\widetilde{V}$ for $\Big[ L_5 \subseteq  Q^5_{10} \subseteq Q^2_{30} \Big]$} \label{figure17}
\begin{equation*}
\input{figure17}
\end{equation*}
\end{figure}
The subvariety $\Sigma_{r_{b_1}} = \Big[ L_4 \subseteq L_5 \subseteq Q^2_{30} \Big]$ has codimension 7. In the inverse image $\pi^{-1}(\Lambda)$ of a general point $\Lambda$ in $\Sigma_{r_{b_1}}$, we have $T^3=\Lambda$, $\;T^2=\Lambda \cap L_5=\Lambda\cap Q^5_{10}$, $\;O^{2, r_{b_1}}={\overline{ Q^{2, sing}_{30}, \Lambda \cap Q^5_{10}}}$, $\;Z^2={\overline{Q^{2, sing}_{30}, \Lambda}}$. We have $Q^{2, sing}_{30} \subseteq O^{2, n_{a_1}} \subseteq L_5$ which can be parameterized by $G(1, 3)$. Then the linear space $T^1$ which satisfies $T^1\subseteq O^{2, n_{a_1}}$ can be parameterized by $G(1, 3)$. On the other hand, $Z^1$ satisfies $Q^{5, sing}_{10} \subseteq Z^1 \subseteq Q^{5}_{10}$ and hence can be parameterized by $OG(1, 5)$. Thus $\dim(\pi^{-1}(\Lambda))=6$ and $\codim (\pi^{-1}(\Sigma_{r_{b_1}}))=1$.
\end{ex}
\medskip


\item[II]
$\Sigma_{n_{a_g}}$ : $\dim(\Lambda \cap L_{n_{a_g}}) = a_g +1$ for some $1 \leq g \leq t$ \\

Depending on $n_{a_g}$, we divide this case into the following two subcases:
\begin{description}
\item[II.A] $g=t$
\item[II.B] $g<t$
\end{description}
\bigskip

\item[II.A] 
$\Sigma_{n_{a_t}}$ : $\dim(\Lambda \cap L_{n_{a_t}}) =a_t+1$ (or equivalently, $\dim(\Lambda \cap L_{n_s})=s+1$) \\

If $r_{b_1}= n_{a_t}$, then $\Sigma_{n_{a_t}}$ corresponds to $\Sigma_{r_{b_1}}$. If $r_{b_1}>n_{a_t}$ then $\Sigma_{r_{b_1}}$ contains $\Sigma_{n_{a_t}}$.  So we assume $r_{b_1}<n_{a_t}$ in the following. In the sequence of $\Sigma_{n_{a_t}}$, the sub-quadric $Q^{r_{k-s}}_{d_{k-s}}$ is replaced with the isotropic linear space $L_{n_{a_t}}$. Consequently, each isotropic linear space $L_\tau$, where $n_{a_t} - \alpha_t +1 \leq \tau \leq n_{a_t}$, is replaced with $L_{\tau -1}$. We have
\begin{align*}
\codim (\Sigma_{r_{b_1}}) & = \alpha_t \left( n_{a_t} - a_t \right) + \sum_{t=1}^{\beta_1} \left( d_{b_1} + x_{b_1} - 2(s+\beta_1) + t-1 \right) \\
& \quad - (\alpha_t + 1) \left( n_{a_t} - a_t -1 \right) - \sum_{t=1}^{\beta_1-1} \left( d_{b_1} + x_{b_1} - 2(s+\beta_1) + t-1 \right)  \\
& = \alpha_t + d_{b_1} + x_{b_1} -s - n_{a_t}  - \beta_1
\end{align*}
The only nontrivial parameterizations in the inverse image $\pi^{-1}(\Lambda)$ of a general point $\Lambda$ in $\Sigma_{n_{a_t}}$ are in the row of $T^t$ and once $T^t$ is fixed, the rest of the row can be determined uniquely. The linear space $T^t$ satisfies ${T^{t-1} \subseteq T^t \subseteq \Lambda \cap L_{n_{a_t}}}$ and hence can be parameterized by $G(\alpha_t, \alpha_t+1)$. Thus we have
\begin{equation*}
\codim (\pi^{-1}(\Sigma_{n_{a_t}}))  = d_{b_1} + x_{b_1} -s - n_{a_t}  - \beta_1
\end{equation*}
Since $d_{b_1} - \beta_1 +1 = d_{k-s}$, this is equivalent to
\[ \codim (\pi^{-1}(\Sigma_{n_s}))  = d_{k-s} + x_{k-s} -s - n_s  - 1 \; . \]
Note that $\codim (\pi^{-1}(\Sigma_{n_{a_t}}))$ may be 1 or larger in this case.

\begin{ex} Let $V=\Big[ L_5 \subseteq Q^2_8 \Big]$, then $\widetilde{V}$ is given by the diagram in Figure \ref{figure18}.
\begin{figure} \caption{$\widetilde{V}$ for $\Big[ L_5 \subseteq Q^2_8 \Big]$} \label{figure18}
\begin{equation*}
\input{figure18}
\end{equation*}
\end{figure}

The subvariety $\Sigma_{n_{a_1}}=\Big[ L_4 \subseteq L_5 \Big]$ has codimension 2. In the inverse image $\pi^{-1}(\Lambda)$ of a general point $\Lambda$ in $\Sigma_{n_{a_1}}$, we have $T^2=\Lambda$ and $Z^1={\overline{ \Lambda, Q^{2, sing}_8}}$. The linear space $T^1$ satisfies ${T^1 \subseteq \Lambda \cap L_5}$ and hence can be parameterized by $G(1, 2)$. Then $O^{1, n_{a_1}}$ is determined uniquely as $O^{1, n_{a_1}}={\overline{Q^{2, sing}_8, T^1}}$. Thus $\dim(\pi^{-1}(\Lambda))=1$ and $\codim (\pi^{-1}(\Sigma_{r_{b_1}}))=2-1=1$.
\end{ex}
\medskip

\begin{ex}
Let $V = \Big[ L_4 \subseteq Q^1_8 \Big]$, an orthogonal Schubert variety in $OG(2, 9)$. The diagram in Figure \ref{figure19} gives the definition of $\widetilde{V}$.
\begin{figure} \caption{$\widetilde{V}$ for $\Big[ L_4 \subseteq Q^1_8 \Big]$} \label{figure19}
\begin{equation*}
\input{figure19}
\end{equation*}
\end{figure}

The subvariety $\Sigma_{n_{a_1}}=\Big[ L_4 \subseteq L_5 \Big]$ has codimension 3. In the inverse image $\pi^{-1}(\Lambda)$ of a general point $\Lambda$ in $\Sigma_{n_{a_1}}$, we have $T^2=\Lambda$ and $Z^1={\overline{ \Lambda, Q^{1, sing}_8}}$. The linear space $T^1$ satisfies ${T^1 \subseteq \Lambda \cap L_4}$ and hence can be parameterized by $G(1, 2)$. Then $O^{1, n_{a_1}}$ is determined uniquely as $O^{1, n_{a_1}}={\overline{Q^{1, sing}_8, T^1}}$. Thus $\dim(\pi^{-1}(\Lambda))=1$ and $\codim (\pi^{-1}(\Sigma_{r_{b_1}}))=3-1=2$.

\end{ex}
\medskip


\item[II.B]
$\Sigma_{n_{a_g}}$ : $\dim(\Lambda \cap L_{n_{a_g}}) = a_g +1$ for some $1 \leq g \leq t-1$ \\

We have already discussed in I.C the case when there is some $Q^{r_{b_h}}_{d_{b_h}}$ in the defining sequence with $r_{b_h}=n_{a_g}$. Also, if there is $Q^{r_{b_h}}_{d_{b_h}}$ in the sequence with $r_{b_h} > n_{a_g}$ then $\Sigma_{n_{a_g}}$ will be contained in $\Sigma_{r_{b_h}}$. So it is sufficient to consider the case when $n_{a_g} > r_{b_h}$ for all $1 \leq h \leq u$, equivalently, when $n_{a_g}>r_{k-s}$.

In the sequence of $\Sigma_{n_{a_g}}$, the isotropic linear space that comes after $L_{n_{a_g}}$ in the sequence of $V$, namely $L_{n_{a_g}+1}$, is replaced with $L_{n_{a_g}}$. Consequently, each isotropic linear space $L_\tau$, where, $n_{a_g}-\alpha_g+1 \leq \tau \leq n_{a_g}$, is replaced with the isotropic linear space of one less dimension, $L_{\tau-1}$. Rule \ref{rule1} applies to the sub-quadric $Q^{r_{b_1}}_{d_{b_1}}$, and consequently to each $Q^{r_i}_{d_i}$ where $b_1 \leq i \leq k-s$; we replace $Q^{r_i}_{d_i}$ with $Q^{r_i+ (n_{a_g} - r_{k-s})}_{d_i- (n_{a_g} - r_{k-s})}$ for all $i$ with $b_1 \leq i \leq k-s$. This increases the value of $x_i$ by $\alpha_g +1$ for all $i$ with $b_1 \leq i \leq k-s$. We have

\begin{align*}
\codim(\Sigma_{n_{a_g}}) & = \alpha_g \left(n_{a_g} - a_g \right) + \alpha_{g+1} \left( n_{a_{g+1}} - a_{g+1} \right) \\
& \quad - ( \alpha_g+1) \left(n_{a_g} - a_g -1 \right) - (\alpha_{g+1}-1) \left( n_{a_{g+1}} - a_{g+1} \right) \\
& \quad + \beta_1(n_{a_g} - r_{k-s}) - \beta_1(\alpha_g +1)
\end{align*}

The only nontrivial parameterizations in the inverse image $\pi^{-1}(\Lambda)$ of a general point $\Lambda$ in $\Sigma_{n_{a_g}}$ are in the $g$-th row of the diagram of $\widetilde{V}$ and once $T^g$ is parameterized the remaining coordinates can be determined uniquely. The linear space $T^g$ satisfies ${T^{g-1} \subseteq T^g \subseteq L_{n_{a_g}}}$ and hence can be parameterized by the Grassmannian $G(\alpha_g, \alpha_g+1)$. Thus we have
\[ \codim (\pi^{-1}(\Sigma_{n_{a_g}})) = n_{a_{g+1}}-n_{a_g}-(a_{g+1}-a_g) +1 +\beta_1(n_{a_g}-\alpha_g-r_{k-s} -1)  .\]
Note that $n_{a_{g+1}}-n_{a_g}\geq  a_{g+1}-a_g+1$ and $n_{a_g}-\alpha_g  \geq k-s+1$ by assumption. Therefore $\codim (\pi^{-1}(\Sigma_{n_{a_g}})) \geq 2$ in this case.

\begin{ex}
Let $V = \Big[ L_2 \subseteq L_4 \subseteq Q^0_9 \Big]$, an orthogonal Schubert variety on $OG(3, 9)$. The diagram in Figure \ref{figure20} gives the definition of $\widetilde{V}$.
\begin{figure} \caption{$\widetilde{V}$ for $\Big[ L_2 \subseteq L_4 \subseteq Q^0_9 \Big]$} \label{figure20}
\begin{equation*}
\input{figure20}
\end{equation*}
\end{figure}

The subvariety $\Sigma_{n_{a_1}}=\Big[ L_1 \subseteq L_2 \subseteq Q^2_7 \Big]$ has codimension 3. In the inverse image $\pi^{-1}(\Lambda)$ of a general point $\Lambda$ in $\Sigma_{n_{a_1}}$, the coordinates $T^3$ and $T^2$ are determined uniquely as $T^3 = \Lambda$ and $T^2 = \Lambda  \cap L_4$. The coordinate $T^1$ satisfies $T^1 \subseteq L_2$ and hence is parameterized by $G(1, 2)$. Thus $\dim(\pi^{-1}(\Lambda))=1$ and $\codim (\pi^{-1}(\Sigma_{r_{b_1}}))=3-1=2$.
\end{ex}
\medskip

\begin{ex} Let $V=\Big[ L_5 \subseteq L_7 \subseteq Q^3_{20} \Big]$, then $\widetilde{V}$ is given by the diagram Figure \ref{figure21}.
\begin{figure} \caption{$\widetilde{V}$ for $\Big[ L_5 \subseteq L_7 \subseteq Q^3_{20} \Big]$} \label{figure21}
\begin{equation*}
\input{figure21}
\end{equation*}
\end{figure}
The subvariety $\Sigma_{n_{a_1}}=\Big[ L_4 \subseteq L_5 \subseteq Q^5_{18} \Big]$ has codimension 3. In the inverse image $\pi^{-1}(\Lambda)$ of a general point $\Lambda$ in $\Sigma_{n_{a_1}}$, we have $T^3=\Lambda$, $T^2=\Lambda \cap L_5=\Lambda \cap L_7$, $Z^1 = {\overline{ Q^{3, sing}_{20}, \Lambda}}$ and $O^{1, n_{a_2}}={\overline{Q^{3, sing}_{20}, \Lambda\cap L_7}}$. The linear space $T^1$ satisfies $T^1 \subseteq L_5 \cap \Lambda$ and hence can be parameterized by $G(1, 2)$. Then $O^{1, n_{a_1}}$ is determined uniquely as $O^{1, n_{a_1}}={\overline{Q^{3, sing}_{20}, T^1}}$. Thus $\dim(\pi^{-1}(\Lambda))=1$ and $\codim (\pi^{-1}(\Sigma_{r_{b_1}}))=3-1=2$.
\end{ex}

\medskip

\item[III]
$\Sigma_{d_{b_h}}$ : $\dim(\Lambda \cap Q^{r_{b_h}}_{d_{b_h}}) = k-b_h +2$ for some $1 \leq h \leq u-1$ \\

This case is similar to the case in II.A. In the sequence of $\Sigma_{d_{b_h}}$, the sub-quadric that comes after $Q^{r_{b_h}}_{d_{b_h}}$ in the sequence of $V$, namely $Q^{r_{b_{h+1}}+\beta_{h+1}-1}_{d_{b_{h+1}}-\beta_{h+1}+1}$, is replaced with $Q^{r_{b_h}}_{d_{b_h}}$. Consequently, each sub-quadric $Q^\rho_\delta$, where $d_{b_h}-\beta_+1 \leq \delta \leq d_{b_h}$ and $r_{b_h} +\beta_h -1 \geq \rho \geq r_{b_h}$, is replaced with $Q^{\rho+1}_{\delta-1}$. Comparing the dimensions of the sequences, we have

\begin{align*}
\codim (\Sigma_{d_{b_h}}) & = \sum_{t=1}^{\beta_h} \left( d_{b_h} + x_{b_h} -2(k-b_h+1) + t-1 \right) \\
& \quad \quad+ d_{b_{h+1}} + x_{b_{h+1}} - 2(k-b_{h+1}+1) + \beta_{h+1} -1\\
& \quad \quad - \sum_{t=1}^{\beta_h+1} \left( d_{b_h} + x_{b_h} - 2(k-b_h+1) + t-1 \right) \\
& = d_{b_{h+1}} + x_{b_{h+1}} - 2(k-b_{h+1}+1)  \\
& \quad\quad - \left( d_{b_h} + x_{b_h} - 2(k-b_h+1) \right)  + \beta_{h+1} + \beta_h +1
\end{align*}

Note that the resulting sequence may contradict condition (9) in Definition \ref{admissible}. Here we give the general rule for remedying this in a general context that applies whenever a sub-quadric is replaced with another sub-quadric.

\begin{rul} \label{rule3}
Given the defining sequence of a restriction variety, consider the modified sequence where a sub-quadric is replaced with another sub-quadric. If $n_j -r_i =1$ for an isotropic linear space $L_{n_j}$ and a sub-quadric $Q^{r_i}_{d_i}$, then let $r_{b_{h_0}} := \max \{ r_{b_h} \leq r_i \}$, and replace $L_{n_j}$ with $L_{r_{b_{h_0}}}$.
\end{rul}

The only nontrivial parameterizations in the inverse image of a general point $\Lambda$ in $\Sigma_{d_{b_h}}$ are in the row of $T^{k-b_h+1}$ and once $T^{k-b_h+1}$ is fixed, the rest of the row can be determined uniquely. The linear space $T^{k-b_h+1}$ satisfies $\Lambda \cap Q^{r_{b_{h-1}}}_{d_{b_{h-1}}} \subseteq T^{k-b_h+1} \subseteq \Lambda \cap Q^{r_{b_h}}_{d_{b_h}}$, and hence can be parameterized by $G(\beta_h, \beta_h+1)$. Thus we have
\begin{align*}
\codim (\pi^{-1}(\Sigma_{d_{b_h}})) & = d_{b_{h+1}} + x_{b_{h+1}} - 2(k-b_{h+1}+1) \\
& \quad\quad- \left( d_{b_h} + x_{b_h} - 2(k-b_h+1) \right) + \beta_{h+1} +1 \\
\end{align*}
which is larger than one by the definition of $\beta_{h+1}$.

\begin{ex} Let $V=\Big[ Q^2_7 \subseteq Q^0_9 \Big]$, then $\widetilde{V}$ is given by the diagram Figure \ref{figure22}.
\begin{figure} \caption{$\widetilde{V}$ for $\Big[ Q^2_7 \subseteq Q^0_9 \Big]$} \label{figure22}
\begin{equation*}
\input{figure22}
\end{equation*}
\end{figure}

The subvariety $\Sigma_{d_{b_1}}=\Big[ Q^3_6 \subseteq Q^2_7 \Big]$ has codimension 3. In the inverse image $\pi^{-1}(\Lambda)$ of a general point $\Lambda$ in $\Sigma_{d_{b_1}}$, we have $T^2=\Lambda$ and $Z^2={\overline{ \Lambda, Q^{2, sing}_7}}$. The linear space $T^1$ satisfies ${T^1 \subseteq \Lambda \cap Q^2_7}$ and hence can be parameterized by $G(1, 2)$. Then $O^{1, r_{b_1}}$ is determined uniquely as $O^{1, r_{b_1}}={\overline{Q^{2, sing}_7, T^1}}$. Thus $\dim(\pi^{-1}(\Lambda))=1$ and $\codim (\pi^{-1}(\Sigma_{r_{b_1}}))=3-1=2$.
\end{ex}

\vspace{1em}
\end{description}

\bigskip

This concludes the computation behind Observation \ref{codimcomp}. The following lemma, which is based on Lemma \ref{coslem}, allow us to give a partial description of the singular locus of $V$.
\smallskip

\begin{lemma} \label{exclem}
A subvariety $\Sigma$ of a restriction variety $V$ satisfying $\codim(\pi^{-1}(\Sigma))>1$ is in the singular locus of $V$.
\end{lemma}

\begin{proof}
Suppose $\codim(\pi^{-1}(\Sigma))>1$ and $\Lambda \in \Sigma$ is a point such that $\pi^{-1}(\Lambda)$ is positive dimensional. If $\Lambda$ is smooth, then in order to check that $\pi$ is a local isomorphism, it suffices to check that the Jacobian does not vanish. Since $\codim(\pi^{-1}(\Sigma))>1$ and the vanishing locus of the Jacobian is a divisor, we conclude that the Jacobian does not vanish. On the other hand, since $\pi$ is not a local isomorphism around $\pi^{-1}(\Lambda)$, we conclude that $\Lambda$ is a singular point.
\end{proof}

\begin{cor} \label{codim1}
Let $V(L_\bullet, Q_\bullet)$ be a restriction variety and $\pi : \widetilde{V}(L_\bullet, Q_\bullet) \to V(L_\bullet, Q_\bullet)$ the resolution of singularities in Theorem \ref{resthm}. The components of the exceptional locus whose images are of the form
\begin{itemize}
\item $\Sigma_{r_{b_h}}$ with $r_{b_h} < n_s$
\item $\Sigma_{n_{a_g}}$ for all $1 \leq g \leq t-1$
\item $\Sigma_{n_s}$ with $d_{k-s} + x_{k-s} -s -n_s  > 2$
\item $\Sigma_{d_{b_h}}$ for all $1 \leq h \leq u-1$
\end{itemize}
are in the singular locus of $V(L_\bullet, Q_\bullet)$. 
\end{cor}

\bigskip

Our results so far give a partial description of the singular locus of a restriction variety, but are inconclusive about the remaining types of loci:
\begin{itemize}
\item $\Sigma_{r_{b_h}}$ with $r_{b_h} \geq n_s$ , and
\item $\Sigma_{n_{a_t}}$ with $d_{k-s} + x_{k-s} -s - n_s  = 2$ .
\end{itemize}
Studying the tangent space of a restriction variety at a point will allow us to observe these loci further in the following.

\bigskip


Now we study the tangent space of a restriction variety $V$ at a point $\Lambda$ starting with the one-step case. We refer the reader to \cite{lakshmibai1} for a different approach to tangent spaces to Schubert varieties, and to \cite{harris} for general information on tangent spaces to Grassmannians.
\smallskip

If $V=L_e$ is an isotropic linear space, the tangent space at a point $v$ can be identified with the quotient $L_e / \Lambda = L_e / <v>$.
\smallskip

If $V=Q^r_d$ is a quadric, the tangent space at a point $v$ can be obtained by evaluating the kernel of the Jacobian of the polynomial $F_{Q^r_d}$ at $v$, and taking the quotient with $v$. More concretely, $F_{Q^r_d}$ can be taken to be
\[ \sum_{i=r+1}^{r+m}x_iy_i \;\; \mbox{if} \;\; d-r=2m \quad \mbox{and} \quad x_{r+m+1}^2+\sum_{i=r+1}^{r+m}x_iy_i \;\; \mbox{if} \;\; d-r=2m+1 \; , \]
and the kernel of the Jacobian is of the form
\[ \kernel \begin{bmatrix}  & \cdots & 0 & y_{r+1} & x_{r+1} & \cdots & \end{bmatrix} \; , \]
where the last nonzero term is $x_{r+m}$ or $2x_{r+m+1}$ depending on the rank of $Q^r_d$. The quotient of the kernel with $\Lambda$ has dimension $d-2$ if evaluated at a smooth point $v$, but has dimension $d-1$ if evaluated at a singular point $v$. 
\smallskip

Now consider a general restriction variety $V$ defined by
\[ L_{n_1} \subseteq \ldots \subseteq L_{n_s} \subseteq Q^{r_{k-s}}_{d_{k-s}} \subseteq \ldots \subseteq Q^{r_1}_{d_1} \; , \]
and let $\Lambda=<v_1, \ldots, v_k>$ be a general point in $V$ with $v_j \in L_{n_j}$ for $1 \leq j \leq s$, and $v_{k-i+1} \in Q^{r_i}_{d_i}$ for $1 \leq i \leq k-s$. An arc $\Gamma(t)$ through $\Lambda$ contained in $V$ is obtained by moving $\Lambda$'s intersection with each step of the sequence inside that step. Explicitly, $\Gamma(t)=<\gamma_1(t), \ldots, \gamma_k(t)>$ where $\gamma_j$ is an arc through $v_j$ contained in $L_{n_j}$ for $1 \leq j \leq s$, and $\gamma_{k-i+1}$ is an arc through $v_{k-i+1}$ contained in $Q^{r_i}_{d_i}$ for $1\leq i \leq k-s$. Therefore the tangent space of $\Gamma$ can be studied by considering the tangent spaces of $\gamma$'s.
\smallskip

The tangent space of $\gamma_j$ for $1 \leq j \leq s$ is given by the quotient
\begin{equation*}
L_{n_j} / \Lambda \; = \; L_{n_j} / <v_1, \ldots, v_j> \; ,
\end{equation*}
which has dimension $n_j-j$.
\smallskip

The arc $\gamma_{k-i+1}$ for $1 \leq i \leq k-s$ lies in the orthogonal complement of $\Lambda \cap Q^{r_i}_{d_i}$, and is contained in $Q^{r_i}_{d_i}$. Let $Q^\prime_i$ be the sub-quadric obtained by specializing the hyperplane section of $Q^{r_i}_{d_i}$ until it is tangent to $Q$ along $\overline{Q^{r_i, sing}_{d_i}, \Lambda \cap Q^{r_i}_{d_i}}$. Note that 
\begin{equation*}
\dim (\overline{Q^{r_i, sing}_{d_i}, \Lambda \cap Q^{r_i}_{d_i}}) = r_i + k-i+1 -x_i \; , \quad \mbox{therefore} \quad Q^\prime_i= Q^{r_i + k-i+1 -x_i}_{d_i - (k-i+1 -x_i)} \; .
\end{equation*}
The tangent space can be identified with the quotient
\begin{equation*}
\kernel \Big[ JF_{Q^\prime_i} \Big]_{v_{k-i+1}} \Big/ \Lambda \; = \; \kernel \Big[ JF_{Q^\prime_i} \Big]_{v_{k-i+1}} \Big/ <v_1, \ldots, v_{k-i+1}> \; ,
\end{equation*}
which has dimension $d_i + x_i -2(k-i+1)$.
\smallskip

Note that for a general point $\Lambda$ in $V$, these dimensions are the expressions appearing in the formula for $\dim V$, hence unsurprisingly $\dim V = \dim T_{\Lambda}V$ at a general point. In other orbits this equality does not necessarily hold, and this is what we inspect for the two types of loci for which our previous results are inconclusive.

\begin{prop}
The loci of type $\Sigma_{r_{b_h}}$ with $r_{b_h} \geq n_s$ are in the singular locus of $V$.
\end{prop}

\begin{proof}
Let $\Lambda$ be a general point in the locus of the form $\Sigma_{r_{b_h}}$ for some $r_{b_h} \geq n_s$. The only arcs affected by the increase of $\dim (\Lambda \cap Q^{r_{b_h}, sing}_{d_{b_h}})$ are the group of $\beta_h$ arcs $\gamma_\iota$ with $b_h - \beta_h +1 \leq \iota \leq b_h$. For each $\iota$, we have $\dim (\Lambda \cap Q^{r_\iota, sing}_{d_\iota})=x_\iota+1$, and
\begin{equation*}
\dim (\overline{Q^{r_\iota, sing}_{d_\iota}, \Lambda \cap Q^{r_\iota}_{d_\iota}}) = r_\iota + k-\iota+1 -(x_\iota +1) \; , \quad \mbox{therefore} \quad Q^\prime_\iota= Q^{r_\iota + k-\iota+1 -x_\iota-1}_{d_\iota - (k-\iota+1 -x_\iota) +1} \; .
\end{equation*}
Consequently, the tangent space $\kernel \Big[ JF_{Q^\prime_\iota} \Big]_{v_{k-\iota+1}} \Big/ \Lambda$ has dimension $d_\iota + x_\iota -2(k-\iota+1) +1$ for each $\iota$ in $b_h - \beta_h +1 \leq \iota \leq b_h$. Hence
\[  \dim T_\Lambda V = \dim V + \beta_h \; ,  \]
which shows that $\Lambda$ is in the singular locus of $V$.
\end{proof}

\bigskip

\begin{prop}
The loci of type $\Sigma_{n_s}$ with $d_{k-s} + x_{k-s} -s -n_s =2$ are in the smooth locus of $V$.
\end{prop}

\begin{proof}
Let $\Lambda$ be a general point in the locus of type $\Sigma_{n_s}$. As a result of $\dim (\Lambda \cap L_{n_s}) =s+1$, both arcs $\gamma_s$ and $\gamma_{s+1}$ are contained in $L_{n_s}$, and hence the tangent space of $\gamma_s$ can be identified with
\[  \kernel \Big[ JF_{Q^\prime_{k-s}} \Big]_{v_s} \Big/ \Lambda \; = \; \kernel \Big[ JF_{Q^\prime_{k-s}} \Big]_{v_s} \Big/ <v_1, \ldots, v_{s+1}>   \; ,  \]
the construction as the one for $\gamma_{s+1}$, but evaluated at $v_s$. We observe the difference in dimensions as
\begin{align*}
\dim T_\Lambda V - \dim V & = \big( d_{k-s} + x_{k-s} -2(s+1) \big) - \big(n_s - s \big) \\
& = d_{k-s} + x_{k-s} -s -n_s -2 \; .
\end{align*}
Hence follows the result.
\end{proof}

\bigskip

In particular, the image of the exceptional locus is not equal to the singular locus in general. The following corollary summarizes the results of this chapter.

\bigskip

\begin{cor} \label{thecor}
Let $V(L_\bullet, Q_\bullet)$ be a restriction variety, $\pi$ the resolution of singularities, $E_\pi$ the exceptional locus of $\pi$, and $\Sigma_\bullet$ the components of $\pi(E_\pi)$ as above. The singular locus of $V$ is the union of
\begin{itemize}
\item $\Sigma_{r_{b_h}}$
\item $\Sigma_{n_{a_g}}$ for all $1 \leq g \leq t-1$
\item $\Sigma_{n_s}$ with $d_{k-s} + x_{k-s} -s -n_s  > 2$
\item $\Sigma_{d_{b_h}}$ for all $1 \leq h \leq u-1$ .
\end{itemize}

\medskip

Equivalently,
\begin{equation*}
V^{sing}=
\begin{cases}
\pi (E_\pi) & \text{if} \quad d_{k-s} + x_{k-s} -s -n_s  > 2 \\
\pi (E_\pi) \setminus \Sigma_{n_s} & \text{if} \quad d_{k-s} + x_{k-s} -s -n_s  = 2 \; .
\end{cases}
\end{equation*}
\end{cor}

\bigskip


\section{The Algorithm and Examples}

We present an algorithm for finding the singular locus of a restriction variety that is based on our study of the exceptional locus of $\pi$. The three rules introduced before will be used in the algorithm, we repeat them here for convenience.

\begin{customthm}{\ref{rule1}}
Given the defining sequence of a restriction variety, consider the modified sequence where an isotropic linear space $L_{n_j}$ is replaced with a smaller dimensional isotropic linear space. If there are sub-quadrics $Q^{r_i}_{d_i}$ in the sequence satisfying $r_i < n_j$, then let $r_{i_0}:= \max \{r_i < n_j \}$, and replace $Q^{r_{i_0}}_{d_{i_0}}$ with $Q^{n_j}_{d_{i_0} - (n_j - r_{i_0})}$.
\end{customthm}

\begin{customthm}{\ref{rule2}}
Given the defining sequence of a restriction variety, consider the modified sequene where a sub-quadric is replaced with an isotropic linear space. If $n_j - r_i =1$ for an isotropic linear space $L_{n_j}$, and a sub-quadric $Q^{r_i}_{d_i}$ in the modified sequence, then let $n_{a_{g_0}}:= \min \{ n_{a_g} \geq n_j \}$, and replace $Q^{r_i}_{d_i}$ with $Q^{n_{a_{g_0}}}_{d_i - (n_{a_{g_0}} - r_i)}$.
\end{customthm}

\begin{customthm}{\ref{rule3}}
Given the defining sequence of a restriction variety, consider the modified sequence where a sub-quadric is replaced with another sub-quadric. If $n_j -r_i =1$ for an isotropic linear space $L_{n_j}$ and a sub-quadric $Q^{r_i}_{d_i}$, then let $r_{b_{h_0}} := \max \{ r_{b_h} \leq r_i \}$, and replace $L_{n_j}$ with $L_{r_{b_{h_0}}}$.
\end{customthm}

\smallskip

\begin{alg} \label{alg}

Let $V$ be defined by the sequence
\[ L_{n_1} \subseteq \ldots \subseteq L_{n_s} \subseteq Q^{r_{k-s}}_{d_{k-s}} \subseteq \ldots \subseteq Q^{r_1}_{d_1} \; , \]
or equivalently, by the partitions
\[  (n_{a_1}^{\alpha_1}, \ldots, n_{a_t}^{\alpha_t}), (d_{b_1}^{\beta_1}, \ldots, d_{b_u}^{\beta_u}), (r_1, \ldots, r_{k-s}) . \]

\begin{enumerate}

\item
Steps for $r_{b_1} \geq n_s$. If $r_{b_1}>x_{b_1}$ then proceed, otherwise $\Sigma_{r_{b_1}} = \varnothing$.
\begin{enumerate}
\item If $r_{b_1}>n_s$ then replace $Q^{r_{k-s}}_{d_{k-s}}$ with $L_{r_{b_1}}$. The resulting sequence gives $\Sigma_{r_{b_1}}$.
\item If $r_{b_1}=n_s$ then replace $Q^{r_{k-s}}_{d_{k-s}}$ with $L_{n_s}$, and replace $L_\tau$, where $n_s - \alpha_t +1 \leq \tau \leq n_s$, with $L_{\tau-1}$. Apply Rule 2. The resulting sequence gives $\Sigma_{r_{b_1}}$.
\item Otherwise $\Sigma_{r_{b_1}} = \varnothing$.
\end{enumerate}

\item
Steps for each $r_{b_h} <n_s$, where $1 \leq h \leq u$. For each $h$, if $r_{b_h}>x_{b_h}$ then proceed, otherwise $\Sigma_{r_{b_h}}=\varnothing$.
\begin{enumerate}
\item If $r_{b_h}<n_s$ and $r_{b_h} \neq n_j$ for any $j$, then let $n_{j_\sharp} = \min \{ n_j \; | \; r_{b_h} < n_j \}$, let $r_{b_\flat} = \max \{ r_{b_h} \; | \; r_{b_h} < n_{j_\sharp} \}$, and replace $L_{n_{j_\sharp}}$ with $L_{r_{b_\flat}}$. Apply Rule 1. The resulting sequence is $\Sigma_{r_{b_h}}$.
\item If $r_{b_h}=n_{a_g} <n_s$ then replace $L_{n_{(a_g+1)}}$ with $L_{n_{a_g}}$, and replace $L_\tau$, where $n_{a_g} -\alpha_g+1 \leq \tau \leq n_{a_g}$, with $L_{\tau -1}$. Apply Rule 1. The resulting sequence is $\Sigma_{r_{b_h}}$.
\item Otherwise $\Sigma_{r_{b_h}} = \varnothing$.
\end{enumerate}

\item Steps for $n_s$. If $n_s >s$ and $d_{b_1}+x_{b_1}-s-n_s >2$ then proceed, otherwise $\Sigma_{n_s}=\varnothing$.
\begin{enumerate}
\item If $n_s > r_{b_1}$ and the proposition $ \big[ \; b_1 \; \mbox{is a special index} \; \big] \wedge \big[ \; 2n_s= d_{b_1}+r_{b_1} \; \big]$ is false, then replace $Q^{r_{k-s}}_{d_{k-s}}$ with $L_{n_s}$, and replace $L_\tau$, where $n_s -\alpha_t+1 \leq \tau \leq n_s$, with $L_{\tau -1}$. Apply Rule 2. The resulting sequence gives $\Sigma_{n_s}$.
\item If $b_1$ is a special index and $2n_s=d_{b_1}+r_{b_1}$ and $k \geq s+2$, then replace $Q^{r_{k-s-1}}_{d_{k-s-1}}$ with $L_{n_s}$, replace $Q^{r_{k-s}}_{d_{k-s}}$ with $L_{n_s-1}$, and replace $L_\tau$, where $n_s -\alpha_t+2 \leq \tau \leq n_s$, with $L_{\tau -2}$. Apply Rule 2. The resulting sequence gives $\Sigma_{n_s}$.
\item Otherwise $\Sigma_{n_s}=\varnothing$.
\end{enumerate}

\item Steps for each $n_{a_g}$, where $1 \leq g \leq t-1$.
\begin{enumerate}
\item If $n_{a_g} > a_g$ and $r_{b_1} <n_{a_g} < n_s$, then replace $L_{n_{(a_g+1)}}$ with $L_{n_{a_g}}$, and replace $L_\tau$, where $n_{a_g}-\alpha_t+1 \leq \tau \leq n_{a_g}$, with $L_{\tau -1}$. Apply Rule 2. The resulting sequence gives $\Sigma_{n_{a_g}}$.
\item Otherwise $\Sigma_{n_{a_g}}=\varnothing$.
\end{enumerate}

\item Steps for each $d_{b_h}$ where $1 \leq h \leq u-1$.
\begin{enumerate}
\item If $d_{b_h} - r_{b_h} -2\beta_h \geq 3$ then replace $Q^{r_{b_{h-1}}+\beta_{h-1}-1}_{d_{b_{h-1}}-\beta_{h-1}+1}$ with $Q^{r_{b_h}}_{d_{b_h}}$, and replace $Q^\rho_\delta$, where $r_{b_h}+\beta_h-1 \geq \rho \geq r_{b_h}$ and $d_{b_h}-\beta_h+1 \leq \delta \leq d_{b_h}$ with $Q^{\rho+1}_{\delta-1}$. Apply Rule 3. The resulting sequence gives $\Sigma_{d_{b_h}}$.
\item Otherwise $\Sigma_{d_{b_h}}=\varnothing$.
\end{enumerate}

\item Take the union of the restriction varieties obtained from the first five steps. The resulting restriction variety gives the singular locus of $V$.

\end{enumerate}
\end{alg}

\smallskip

Here are some examples illustrating Algorithm \ref{alg} in a few different cases. We refer the reader to \cite{billeylakshmibai} for the permutation notation, and for more examples on singularities of Schubert varieties.
\smallskip

\begin{ex}
Let $V$ be defined by the sequence $\Big[ Q^3_6 \subseteq Q^0_9 \Big]$. This is a Schubert variety in $OG(2, 9)$. The locus $\Sigma_{r_{b_1}} = \Big[ L_3 \subseteq Q^0_9 \Big]$ is obtained by step (1)(a) in Algorithm \ref{alg}. Note that the locus $\Sigma_{d_{b_1}}$ does not exist since $d_{b_1}-r_{b_1}-2\beta_1 \ngeq 3$. Therefore
\[ V^{sing} = \Big[ L_3 \subseteq Q^0_9 \Big] \; . \]
Equivalently, in permutation notation we have
\[ (968753241)^{sing} = (938654271) \; . \]
\end{ex}
\smallskip

\begin{ex}
Let $V$ be the Schubert variety in $OG(2, 9)$ defined by $\Big[ Q^2_7 \subseteq Q^0_9 \Big]$. Using the steps (1)(a) and (5)(a), we have
\[ V^{sing} = \Big[ L_2 \subseteq Q^0_9 \Big] \cup \Big[ Q^3_6 \subseteq Q^2_7 \Big] \; , \]
equivalently, in permutation notation
\[ (978654231)^{sing} = (9276543811) \cup (769852143) \; .\]
\end{ex}
\smallskip

\begin{ex}
Let $V$ be the Schubert variety in $OG(3, 9)$ defined by $\Big[ L_3 \subseteq Q^3_6 \subseteq Q^0_9 \Big]$. Using the step (1)(b), we have
\[ V^{sing} = \Big[ L_2 \subseteq L_3 \subseteq Q^0_9 \Big] \; , \]
equivalently, in permutation notation
\[ (963852741)^{sing} = (932654871) \; . \]
\end{ex}
\smallskip

\begin{ex}
Let $V$ be the Schubert variety in $OG(2, 8)$ defined by $\Big[ L_3 \subseteq Q^1_7 \Big]$. The locus $\Sigma_{r_{b_1}}$ is obtained by step (2)(a). Rule 1 replaces $Q^1_7$ with the union of $L_4$ and $L_4^\prime$. Therefore the locus $\Sigma_{r_{b_1}}$ is the union of $\Big[ L_1 \subseteq L_4 \Big]$ and $\Big[ L_1 \subseteq L_4^\prime \Big]$. Furthermore, step (3)(a) gives the locus $\Big[ L_2 \subseteq L_3 \Big]$. Hence
\[ V^{sing} = \Big[ L_2 \subseteq L_3 \Big] \cup \Big[ L_1 \subseteq L_4 \Big] \cup \Big[ L_1 \subseteq L_4^\prime \Big] \; , \]
equivalently, in permutation notation
\[ (73845162)^{sing} = (32854176) \cup (41763285) \cup (51736284) \; . \]
\end{ex}
\smallskip

\begin{ex}
Let $V$ be the Schubert variety in $OG(3, 9)$ defined by $\Big[ L_2 \subseteq L_4 \subseteq Q^2_7 \Big]$. The locus $\Sigma_{r_{b_1}} = \Big[ L_1 \subseteq L_2 \subseteq L_4 \Big]$ is obtained by applying step (4)(a) and in particular Rule 1. Note that $\Sigma_{n_{a_2}}$ is in the smooth locus of $V$ since $d_{b_1}+x_{b_1}-s-n_s=2$. We have
\[ V^{sing} = \Big[ L_1 \subseteq L_2 \subseteq L_4 \Big] \; , \]
equivalently, in permutation notation
\[ (742951863)^{sing} = (421753986) \; . \]
\end{ex}
\smallskip

\begin{ex}
Let $V$ be the Schubert variety in $OG(3, 9)$ defined by $\Big[ L_3 \subseteq Q^1_7 \subseteq Q^0_9 \Big]$. The locus $\Sigma_{n_{a_1}} = \Big[ L_2 \subseteq L_3 \subseteq Q^0_9 \Big]$ is obtained by step (3)(a). We have
\[ V^{sing} = \Big[ L_2 \subseteq L_3 \subseteq Q^0_9 \Big] \; , \]
equivalently, in permutation notation
\[ (983654721)^{sing} = (932654871) \; . \]
\end{ex}
\smallskip

\begin{ex}
Let $V$ be the Schubert variety in $OG(3, 9)$ defined by $\Big[ L_4 \subseteq Q^2_7 \subseteq Q^1_8 \Big]$. The locus $\Sigma_{r_{b_1}} = \Big[ L_1 \subseteq L_4 \subseteq Q^1_8 \Big]$ is obtained by applying step (2)(a). Furhermore, step (3)(a) is applied to obtain the locus $\Sigma_{n_{a_1}} = \Big[ L_2 \subseteq L_3 \subseteq L_4 \Big]$. Thus
\[ V^{sing} = \Big[ L_1 \subseteq L_4 \subseteq Q^1_8 \Big] \cup \Big[ L_2 \subseteq L_3 \subseteq L_4 \Big] \; , \]
equivalently, in permutation notation
\[ (874951632)^{sing} = (841753962) \cup (432951876) \; . \]
\end{ex}
\smallskip

\begin{ex}
Let $V$ be the Schubert variety in $OG(3, 9)$ defined by $\Big[ L_2 \subseteq L_4 \subseteq Q^0_9 \Big]$. By steps (4)(a) and (3)(a), we have
\[ V^{sing} = \Big[ L_1 \subseteq L_2 \subseteq Q^2_7 \Big] \cup \Big[ L_2 \subseteq L_3 \subseteq L_4 \Big] \; , \]
equivalently, in permutation notation
\[ (942753861)^{sing} = (721654983) \cup (432951876) \; . \]
\end{ex}
\smallskip

\begin{ex}
Let $V$ be the Schubert variety in $OG(4, 9)$ defined by $\Big[ L_2 \subseteq L_4 \subseteq Q^2_7 \subseteq Q^0_9 \Big]$. Step (2)(b) is applied to obtain the locus $\Sigma_{d_{b_1}} = \Big[ L_1 \subseteq L_2 \subseteq Q^3_6 \subseteq Q^2_7 \Big]$. Note that $\Sigma_{n_{a_1}}$ is contained in $\Sigma_{d_{b_1}}$, and $\Sigma_{n_{a_2}}$ is contained in the smooth locus of $V$ since $d_{b_1}+x_{b_1}-s-n_s=2$. Hence
\[ V^{sing} = \Big[ L_1 \subseteq L_2 \subseteq Q^3_6 \subseteq Q^2_7 \Big] \; , \]
equivalently, in permutation notation
\[ (974258631)^{sing} = (762159843) \; . \]
\end{ex}
\smallskip

\begin{ex}
Let $V$ be the restriction variety in $OG(6, 21)$ defined by the sequence
\[ \Big[ L_3 \subseteq L_8 \subseteq L_9 \subseteq Q^6_{12} \subseteq Q^5_{13} \subseteq Q^1_{20} \Big] \; . \]
The loci $\Sigma_{r_{b_1}}$ and $\Sigma_{r_{b_2}}$ are obtained by applying step (2)(a). When applied to $Q^1_{20}$, we have
\[ \Sigma_{r_{b_2}} = \Big[ L_1 \subseteq L_8 \subseteq L_9 \subseteq Q^6_{12} \subseteq Q^5_{13} \subseteq Q^3_{18} \Big] \; , \]
and when applied to $Q^5_{13}$, the sub-quadric $Q^6_{12}$ is replaced with the isotropic linear subspaces $L_9$ and $L_9^\prime$. Thus
\[ \Sigma_{r_{b_1}} = \Big[ L_3 \subseteq L_5 \subseteq L_8 \subseteq L_9 \subseteq Q^5_{13} \subseteq Q^1_{20} \Big]  \cup \Big[ L_3 \subseteq L_5 \subseteq L_8 \subseteq L_9^\prime \subseteq Q^5_{13} \subseteq Q^1_{20} \Big] \; . \]
Applying step (4)(a) gives the locus
\[ \Sigma_{n_{a_1}} = \Big[ L_2 \subseteq L_3 \subseteq L_9 \subseteq Q^6_{12} \subseteq Q^5_{13} \subseteq Q^1_{20} \Big] \; . \]
Since $b_1$ is a special index and $2n_s = d_{b_1} + r_{b_1}$, step (3)(b) is applied to obtain the locus
\[ \Sigma_{n_{a_2}} = \Big[ L_3 \subseteq L_6 \subseteq L_7 \subseteq L_8 \subseteq L_9 \subseteq Q^1_{20} \Big] \; . \]
As a result, we have
\[ V^{sing} = \Sigma_{n_{a_1}} \cup \Sigma_{n_{a_2}} \cup \Sigma_{r_{b_1}} \cup \Sigma_{r_{b_2}} \; . \]
\end{ex}

\bigskip

\bibliography{singrest5}
\bibliographystyle{spmpsci}

\end{document}